%% file: Inertialess_Limit_rigorous.tex
\documentclass[abstracton, paper=a4, fontsize=11pt, DIV=12]{scrartcl}

\usepackage[obeyspaces,hyphens,spaces]{url}

\usepackage[T1]{fontenc}
\usepackage{verbatim}
\usepackage{amsmath,amsthm,amssymb,amsfonts}
\usepackage{amsthm}
\usepackage{mathrsfs}
\usepackage[utf8]{inputenc}
\usepackage{enumitem}
\usepackage{caption}
\usepackage{mathtools}
\usepackage{esint}
\usepackage{textcomp}
\usepackage[backend=bibtex,firstinits,maxnames=4,style=alphabetic,isbn=false, doi=false, eprint= false, url= false]{biblatex}
\usepackage{csquotes}
\usepackage{nicefrac}
\usepackage{bm}

    \makeatletter
    \renewenvironment{proof}[1][\proofname]{%
      \par\pushQED{\qed}\normalfont%
      \topsep6\p@\@plus6\p@\relax
      \trivlist\item[\hskip\labelsep\bfseries#1\@addpunct{.}]%
      \ignorespaces
    }{%
      \popQED\endtrivlist\@endpefalse
    }
    \makeatother

\newcommand{\dist}{\operatorname{dist}}
\newcommand{\dv}{\operatorname{div}}
\newcommand{\op}{\operatorname}
\newcommand{\IR}{\mathbb{R}}
\newcommand{\IN}{\mathbb{N}}

\newcommand{\tr}{\operatorname{tr}}
 
\newcommand{\dd}{\, \mathrm{d}}

\newcommand{\supp}{\operatorname{supp}}

\newcommand{\eps}{\varepsilon}

\newtheorem{theorem}{Theorem}[section]
\newtheorem{lemma}[theorem]{Lemma}

\newtheorem{proposition}[theorem]{Proposition}

\theoremstyle{definition}

\newtheorem{remark}[theorem]{Remark}

\bibliography{bibliography.bib}

\mathtoolsset{showonlyrefs}

\begin{document}

\input{Introduction}
\input{WellPosedness.tex}
\input{A_priori_estimates}
\input{ParticleVelocityFluidVelocity}

\input{Convergence}
\input{ConvergenceForArbitraryTimes}

\section*{Acknowledgement}

The author acknowledges support through the CRC 1060, the mathematics of emergent effects, of the University of Bonn,
that is funded through the German Science Foundation (DFG).


\printbibliography

\end{document}

%% file: Introduction.tex
\begin{center}
{\Large The inertialess limit of particle sedimentation modeled by the Vlasov-Stokes equations}

\bigskip

Richard M. H\"{o}fer\footnote{University of Bonn, Institute For Applied Mathematics. Endenicher Allee 60, 53115 Bonn, Germany. 
\newline Email: hoefer@iam.uni-bonn.de, Phone: +49 228 735602}

\vspace{2mm}

\today

\end{center}

\begin{abstract}
We study the Vlasov-Stokes equations which macroscopically model the sedimentation of a cloud of particles in a fluid,
where particle inertia are taken into account but fluid inertia are assumed to be negligible.
We consider the limit when the inertia of the particles tends to zero, and obtain convergence of the 
dynamics to the solution of an associated inertialess system of equations. This system coincides with the model that can be derived
as the homogenization limit of the microscopic inertialess dynamics.
\end{abstract}

\section{Introduction}
We consider the sedimentation of a cloud of identical spherical particles suspended in a fluid subject to gravitation.
It is assumed that the suspension is sufficiently dilute such that collisions of particles do not play a role.
Furthermore, we neglect inertial forces of the fluid, i.e., the fluid is modeled by a Stokes equation, but particle inertia are taken into account.
These assumptions are justified if the Reynolds number is much smaller than the Stokes numbers which is the case for very small particles in gases. We refer to \cite{Koc90} for the details
of the microscopic model and a discussion about the regime of validity.

Let a nonnegative function $f(t,x,v)$ describe the number density of particles at time $t$ and position $x \in \IR^3$ with velocity $v \in \IR^3$.
We denote the position density and current by
\begin{align}
	\rho(t,x) &:= \int_{\IR^3} f(t,x,v) \dd v,  \label{eq:defRho}\\
	j(t,x) := \rho(t,x) \bar{V}(t,x) &:= \int_{\IR^3} f(t,x,v) v \dd v. \label{eq:defJVbar}
\end{align}
Here, the mean velocity $\bar{V}$ is defined to be zero in the set $\{\rho = 0\}$.
As a model for the macroscopic dynamics, we consider the so-called Vlasov-Stokes equations, a Vlasov equation for the particles coupled with Brinkman equations for the fluid,
\begin{equation}
		\label{eq:VlasovStokes0}
\begin{aligned}
	\partial_t f + v \cdot \nabla_x f + \lambda \dv_v \left(\hat{g} f + \frac{9}{2} \gamma  (u-v) f  \right) &= 0, \qquad f(0,\cdot,\cdot) = f_0, \\
	- \Delta u + \nabla p + 6 \pi \gamma \rho(u-\bar V) &= 0, \qquad \dv u = 0.
\end{aligned}
\end{equation}
Here, $u$ and $p$ are the fluid velocity and pressure respectively, $\hat{g} = g/|g|$ with $g$ being the gravitational acceleration, and
$\lambda$ and $\gamma$ are constants that will be discussed below. 
The first equation expresses that the forces acting on the particles are the gravitation and the drag exerted by the fluid.
The Brinkman equations are Stokes equations with a force term that arises from the same drag.

A rigorous derivation of these macroscopic equations from the microscopic dynamics has not been achieved yet, a formal derivation can be found in \cite{Koc90}. 
In the quasi-static case, the Brinkman equations have been established in
\cite{Al90a}, \cite{DGR08}. 
Using this, the Vlasov-Stokes equations \eqref{eq:VlasovStokes0} can be formally derived from the microscopic dynamics after non-dimensionalizing.
The constants $\lambda$ and $\gamma$ are given
by
\begin{align}
	\lambda = \frac{\mu^2 }{\rho_p (\rho_p - \rho_f) \phi^2 |g| L^3}, \qquad
	\gamma = \frac{\phi L^2}{R^2},
\end{align}
where $\mu$ is the fluid viscosity, $\rho_p$ and $\rho_f$ are the particle and fluid mass density respectively,
$\phi$ is the volume fraction of the particles, $L$ is the diameter of the cloud of particles, and $R$ the radius of the particles.
The constant $\gamma$ determines the interaction strength between fluid and particles. The quantity $(\lambda \gamma)^{-1}$ is known as the Stokes number
and determines the strength of the inertial forces.
For definiteness, we assume $\rho_p > \rho_f$ such that $\lambda >0$. Then, the larger $\lambda$, the less important
inertial effects become. For a more detailed discussion of these parameters as well as a formal derivation of the system \eqref{eq:VlasovStokes0},
we refer to \cite{Hof16}.

For similar equations as \eqref{eq:VlasovStokes0}, global well-posedness has been proven in \cite{Ham98} and \cite{BDGM09}.
In \cite{Jab00}, the author considers the inertialess limit of the system, where the fluid velocity $u$ in \eqref{eq:Vlasov Stokes} is replaced by
a force term $F[\rho,j]$ that is given by a convolution operator which is more regular than the Stokes convolution operator.
In \cite{Gou01}, similar limits are studied for a one dimensional model without gravity and including inertial forces on the fluid.
In \cite{GP04}, the authors consider limits of high and low inertia of the system of a Vlasov equation without gravity and with a given random fluid velocity field.
Similar systems that include Brownian motion of the particles and their limits have been studied among others in \cite{CP83}, 
\cite{GJV04a}, \cite{GJV04b} \cite{CG06}, and \cite{GHMZ10}.

\subsection{Main result}

We are interested in the limit $\lambda \to \infty$, which corresponds to inertialess particles.
For the ease of notation we drop all the other constants and consider the system
\begin{equation}
	\label{eq:Vlasov Stokes}
\begin{aligned}
	\partial_t f + v \cdot \nabla_x f + \lambda \dv_v \left(g f +   (u-v) f  \right) &= 0, \qquad f(0,\cdot,\cdot) = f_0, \\
	- \Delta u + \nabla p + \rho(u-\bar V) &= 0, \qquad \dv u = 0.
\end{aligned}
\end{equation}

For inertialess particles, the following macroscopic equation has been proven in \cite{Hof16} to be the homogenization limit of
many small particles.
\begin{equation}
	\label{eq:limitEquation}
\begin{aligned}
	 \partial_t \rho_\ast + \left( g + u_\ast \right) \cdot \nabla \rho_\ast &= 0, 
	 \qquad \rho_\ast(0,\cdot) = \rho_0 := \int_{\IR^3} f_0 \dd v,\\
	 - \Delta u_\ast + \nabla p& =   {g} \rho_\ast, \qquad \dv u_\ast = 0 .
\end{aligned}
\end{equation}
Moreover, well-posedness of this system has been proven in \cite{Hof16}.

In these equations, particles are described by their position density $\rho_\ast$  only, because their velocity is 
the sum of the fluid velocity $u$ and the constant $g$ which is the direct effect due to gravitation. 

The main result of this paper is the following theorem.
\begin{theorem}
	\label{th:main}
	Assume $f_0 \in W^{1,\infty}(\IR^3\times\IR^3)$ is compactly supported. Then, for $\lambda > 0$, there exists
	a unique solution $(f_\lambda,u_\lambda)$ to \eqref{eq:Vlasov Stokes}.
	Let $(\rho_\ast,u_\ast)$ be the unique solution to \eqref{eq:limitEquation}. 
	Then, for all $T > 0$, and all $\alpha < 1$
	\begin{align}
		\rho_\lambda &\to \rho_\ast \quad \text{in} ~ C^{0,\alpha}((0,T) \times \IR^3), \\
		u_\lambda &\to u_\ast \quad  \text{in} ~ L^\infty((t,T) ; W^{1,\infty}(\IR^3)) ~  \text{and in} ~ L^1((0,T) ; W^{1,\infty}(\IR^3)).
	\end{align}
\end{theorem}

Formally, for large values of $\lambda$, the first equation in \eqref{eq:Vlasov Stokes}
forces the particle to attain the velocity $g + u(t,x)$, i.e., the density $f(t,x,v)$  concentrates around $g + u(t,x)$. 
Using that and integrating the first equation in \eqref{eq:Vlasov Stokes} in $v$ leads to the first equation in \eqref{eq:limitEquation}.
Moreover, $\bar{V}$ in the fluid equation in \eqref{eq:Vlasov Stokes} can formally be replaced by $g + u(t,x)$, which leads to the fluid equation
in \eqref{eq:limitEquation}.

Formally, the adjustment of the particle velocities described above happens in times of order  $1/\lambda$.
In fact, the process is more complicated as the fluid velocity changes very fast in this time scale as well. In other words,
there is a boundary layer of width $1/\lambda$
at time zero for the convergence of the fluid (and particle) velocity. This is the reason, why the convergence $u_\lambda \to u_\ast$ can only hold
 uniformly on time intervals $(t,T)$ for $t \geq 0$ as stated in the theorem.
The particles, however, do not move significantly in times of order $1/\lambda$. Thus, there is no boundary layer in the convergence $\rho_\lambda \to \rho_\ast$.

\subsection{Idea of the proof}
We introduce the kinetic energy of the particles
\[
	E(t) := \int_{\IR^3 \times \IR^3} \hspace{-1em} |v|^2 f \dd x \dd v .
\]
Using the Vlasov-Stokes equations \eqref{eq:Vlasov Stokes} yields the following energy identities for the fluid velocity and the particle energy (cf. Lemma \ref{lem:WellPosednessFluid} and Lemma \ref{lem:aPrioriWellPosedness}).
\begin{align}
\label{eq:energyFluid}
\|\nabla u\|_{L^2(\IR^3)}^2 + \| u \|_{L^2(\rho)} &= (u,j)_{L^2(\IR^3)} \leq \|\bar{V} \|^2_{L^2_\rho} \leq  E, \\
\label{eq:energyParticles}
	\frac{1}{2} \frac{d}{d t} E
	&=\lambda \left( g \cdot \int_{\IR^3 \times \IR^3} j \dd x 
	-  \int_{\IR^3 \times \IR^3} \hspace{-1em}    (u - v)^2 f \dd x \dd v - \|\nabla u \|^2_{L^2(\IR^3)} \right).
\end{align}
Here and in the following, the weighted $L^p$-norm is defined by
\[
	\|h\|_{L^p_\rho}^p := \int_{\IR^3} |h|^p \rho \dd x.
\]
As expected,  equation  \eqref{eq:energyParticles} shows that there is loss of energy due to friction (friction between the particles and the fluid as well as friction inside of the fluid),
but the gravity pumps energy into the system (if we assume $g \cdot \int_{\IR^3 \times \IR^3} j \dd x > 0$, which at least after some time should be the case). Note that the Vlasov-Stokes equations \eqref{eq:Vlasov Stokes} also imply that the mass of the particles $\|\rho\|_{L^1(\IR^3)}$
is conserved.

To analyze solutions to the Vlasov equation in \eqref{eq:Vlasov Stokes}, we look at the characteristic curves  $(X,V,Z)(s,t,x,v)$ starting at time $t$ at position $(x,v) \in \IR^3 \times \IR^3$, where $Z$ denotes the value of the solution $f$ along the characteristic curve.
\begin{equation}
	\label{eq:characteristics}
\begin{aligned}
	\partial_s{X} &= V, \qquad  &&X(t,t,x,v) = x, \\
	\partial_s{V} &= \lambda (g + u(s,{X}) - V(s,t,x,v)), \qquad &&V(t,t,x,v) = v, \\
	\partial_s{Z} &= 3 \lambda Z, \qquad &&Z(t,t,x,v) = f(t,x,v).
\end{aligned}
\end{equation}
By the standard theory, any solution $f \in W^{1,\infty}((0,T) \times \IR^3 \times \IR^3)$ with $u \in L^\infty((0,T);W^{1,\infty}(\IR^3))$
is of the form
\begin{equation}
	\label{eq:fByCharacteristics}
	f(t,x,v) = e^{3 \lambda t} f_0(X(0,t,x,v),V(0,t,x,v)).
\end{equation}

Using the characteristics as well as estimates based on the energy identities \eqref{eq:energyFluid} and \eqref{eq:energyParticles} and regularity theory of Stokes equations,
we prove global well-posedness of the Vlasov-Stokes equations \eqref{eq:Vlasov Stokes} for compactly supported initial data $f_0 \in W^{1,\infty}(\IR^3 \times \IR^3)$.
A similar approach based on an analysis of the characteristics has been used to prove existence of solutions to the Vlasov-Poisson equations in \cite{BD85}, \cite{Pfa92}, and \cite{Sch91} (see also \cite{Gla96}).
From the PDE point of view, the electrostatic potential appearing in the Vlasov-Poisson equation is similar to the fluid velocity in the Vlasov-Stokes equations.
However, in the Vlasov-Poisson equations, the force acting on the particles is the gradient of the electrostatic potential.
whereas in the Vlasov-Stokes equations, only the fluid velocity itself contributes.
This makes it possible to prove existence (and also uniqueness) in a much simpler way for the Vlasov-Stokes equations.

In order to prove the convergence in Theorem \ref{th:main}, the starting point is integrating the characteristics which yields
\begin{equation}
	\label{eq:integrateCharacteristics}
	V(t,0,x,v) - V(0,0,x,v) = \lambda \left(\int_0^t u_\lambda(s,X(s,0,x,v)) + g \dd s + X(0,0,x,v) - X(t,0,x,v) \right).
\end{equation}
Thus,
\begin{equation}
	\label{eq:almostTransportedByU}
	\left| X(t,0,x,v) - x - \int_0^t u_\lambda(s,X(s,0,x,v)) + g \dd s \right| \leq \frac{|V(t,0,x,v) - v|}{\lambda}.
\end{equation}
Therefore, provided the speed of the particles does not blow up, we see that for large values of $\lambda$ the particles are almost transported by the fluid plus the gravity. Clearly, this is also what happens for solutions to the limit inertialess equations \eqref{eq:limitEquation}.

In order to show that $u_\lambda$ is close to $u$, we introduce a fluid velocity $\tilde{u}_\lambda$ which can be viewed as intermediate between
$u_\lambda$ and $u_\ast$ by
\begin{equation}
	\label{eq:uTilde}
	- \Delta \tilde{u}_\lambda + \nabla p_\lambda =  g \rho_\lambda, \qquad \dv \tilde{u}_\lambda = 0.
\end{equation}
In order to prove smallness of $u_\lambda - \tilde{u_\lambda}$, one needs estimates on $\rho_\lambda$ and $u_\lambda$ that are uniform in $\lambda$,
which are more difficult to obtain than those that we use in the proof of well-posedness. 
Indeed, in view of the energy identity for the particles \eqref{eq:energyParticles}, any naive estimate based on that equation will
blow up as $\lambda \to \infty$. However, as the first term is linear in the velocity and the other terms (which have a good sign) are quadratic,
the energy $E$ cannot exceed a certain value as long as the particle density $\rho$ is not too concentrated (cf. Lemma \ref{lem:boundsU}).
In other words, if the energy is high enough, the quadratic friction terms will prevail over the linear gravitation terms and therefore 
prevent the energy from increasing further.
However, if concentrations of the particle density occur, the particles essentially fall down like one small and heavy particle,
leading to large velocities. Indeed, the terminal velocity of a spherical particle of radius $R$ in a Stokes fluid at rest is
\[
	V = \frac{2}{9}\frac{\rho_p - \rho_f}{\mu} g R^2.
\]

In order to rule out such concentration effects, we use again the representation of $f$ in \eqref{eq:fByCharacteristics} obtained from the characteristics.
Indeed, computing $\rho$ by taking the integral over $v$ in \eqref{eq:fByCharacteristics}, we can show that the prefactor $e^{3 \lambda t}$ 
in that formula is canceled due to concentration of $f$ in velocity space in regions of size $e^{-\lambda t}$ as long as we control $\nabla u$ in a suitable way (cf. Lemma \ref{lem:biLipschitzV}).
As $\nabla u$ is controlled by $E$ due to the energy identity \eqref{eq:energyFluid}, this enables us to get uniform estimates
for both $u$, $\nabla u$, and $\rho$ for small times.

It turns out that also estimates on derivatives of $\rho$ are needed to prove smallness of $u_\lambda - \tilde{u_\lambda}$.
These are provided by a more detailed analysis of the characteristics.

\subsection{Plan of the paper}

The rest of the paper is organized as follows.

In Section 2, we prove global well-posedness of the Vlasov Stokes equations \eqref{eq:Vlasov Stokes}, based on energy estimates, 
analysis of the characteristics, and a fixed point argument.

In Section 3, we derive a priori estimates that are uniform in $\lambda$ for small times by analyzing the characteristics more carefully.
In particular we prove and use that the supports of the solutions concentrate in the space of velocities.

In Section 4.1, we use those a priori proven in Section 4 to show that the fluid velocity $u_\lambda$ is close to 
the intermediate fluid velocity $\tilde{u}_\lambda$ defined in \eqref{eq:uTilde} as $\lambda \to \infty$.
In Section 4.2, we prove the assertion of the main result, Theorem \ref{th:main}, up to times where we have uniform a priori estimates.
This follows from compactness due to the a priori estimates and convergence of averages of $\rho_\lambda$ on small cubes, which we prove using again the characteristic equations.
In Section 4.3, we finish the proof of the main result, Theorem \ref{th:main}, by extending the
a priori estimates from Section 3 to arbitrary times. This is done by using both the a priori estimates and the convergence for small times.

%% file: WellPosedness.tex
\section{Global well-posedness of the Vlasov-Stokes equations}
In this section, we write $C$ for any constant that depends only on the initial datum.
Any additional dependencies are denoted by arguments of $C$, e.g. $C(\lambda t)$ is a constant that
depends only on $\lambda t$ and the initial datum. We use the convention that
$C$ is monotone in all its arguments.

\subsection{Estimates for the fluid velocity}

\begin{lemma}
	\label{lem:WellPosednessFluid}
 Let $g \in L^\infty( \IR^3 \times \IR^3)$ be nonnegative, and assume $Q > 0$ is such that
  $\supp g \subset B_Q(0) \subset \IR^3\times \IR^3$.
  Let
 \begin{align}
 	\rho (x) &:= \int_{\IR^3} g(x,v) \dd v, \\
 	 j(x) := \rho \bar{V} &:= \int_{\IR^3} g(x,v) v \dd v, \\
 	E &:= \int_{\IR^3 \times \IR^3} \hspace{-1em} g(x,v) |v|^2 \dd x \dd v.\\
 \end{align} 
Then there exists a unique weak solution $u \in W^{1,\infty}(\IR^3)$ to the Brinkman equation 
 \[
 	- \Delta u + \nabla p + \rho u = j.
 \]
Moreover, 
 \begin{align}
 	\label{eq:energyEstimateFluid}
 	\|\nabla u\|_{L^2(\IR^3)}^2 + \| u \|_{L^2_\rho(\IR^3)} &= (u,j)_{L^2(\IR^3)} \leq \|\bar{V} \|^2_{L^2_\rho(\IR^3} \leq  E, \\
	\label{eq:L^inftyFluid} 	
 	\| u \|_{L^\infty(\IR^3)} & \leq C(\|g\|_{L^\infty(\IR^3},\|g\|_{L^1(\IR^3)},E) (1 + Q), \\
 	\label{eq:LipschitzFluid}
 	\|u \|_{W^{1,\infty}(\IR^3)} & \leq C(Q,E) \|g\|_{L^\infty(\IR^3)}.
 \end{align}
\end{lemma}

\begin{proof}
Existence and uniqueness of weak solutions in $\dot{H}^1(\IR^3) := \{ w \in L^{6}(\IR^3) \colon \nabla w \in L^2(\IR^3) \}$
follows from the Lax-Milgram theorem.

In the following, we write $\|\cdot\|_q$ instead of $\|\cdot\|_{L^q(\IR^3)}$ and $\|\cdot\|_{L_\rho^q}$ instead of $\|\cdot\|_{L_\rho^q(\IR^3)}$.
Testing the Brinkman equation with $u$ itself yields
	\begin{equation}
		\label{eq:energyIdentity}
		\|\nabla u\|_2^2 + \| u \|^2_{L^2_\rho} = (j,u)_{L^2(\IR^3)} \leq \|u\|_{L^2_\rho} \|\bar{V}\|_{L^2_\rho}.
	\end{equation}
	By the Cauchy-Schwarz inequality
	\begin{equation}
		\label{eq:VBarByE}
		\bar{V}^2 \rho = \frac{\left(\int_{\IR^3} g(x,v) v \dd v\right)^2}{\int_{\IR^3} g(x,v) \dd v} \leq \int_{\IR^3} g(x,v) v^2 \dd v.
	\end{equation}
	Hence,
	\[
		\| u \|^2_{L^2(\rho)} \leq \|\bar{V}\|_{L^2(\rho)} \leq E.
	\]
	Using again \eqref{eq:energyIdentity} yields \eqref{eq:energyEstimateFluid}. 
	Using the critical Sobolev embedding, we have
	\begin{equation}
		\label{eq:uInL^6}
		\|u\|_6^2 \leq C \| \nabla u \|_2^2 \leq C E.
	\end{equation}
	Moreover, we can use this Sobolev inequality in \eqref{eq:energyEstimateFluid} to get
	\[
		\|u\|_6^2 \leq  C\|u\|_6 \|j\|_{{6/5}}.
	\]
	Using the definition of $Q$  yields $\|j\|_{6/5} \leq C(Q) \|g\|_{\infty}$ and therefore
	\begin{equation}
		\label{eq:FluidEstimatesLinear}
		\| \nabla u \|_{2} + \|u\|_{6} \leq C(Q) \|g\|_{\infty}
	\end{equation}
	
Standard regularity theory for the Stokes equation (see \cite{Ga11}) implies
\begin{equation}
	\label{eq:ellipticRegularity}
	\| \nabla^2 u\|_{q} \leq  C\|\rho u \|_q + C\|j\|_q.
\end{equation}
for all $ 1 < q < \infty$.
In order to prove \eqref{eq:LipschitzFluid}, we use \eqref{eq:ellipticRegularity} and  \eqref{eq:uInL^6} to get
\[
	\| \nabla^2 u\|_{6} \leq  C\|\rho u \|_6 + C\|j\|_6 \leq  C\|\rho\|_\infty \| u \|_6 + C\|j\|_6 \leq C(E, Q) \|g\|_\infty.
\]
Hence, by Sobolev embedding and \eqref{eq:FluidEstimatesLinear}
\[
	\| \nabla u\|_{\infty} \leq C \| \nabla^2 u\|_{6} + C \| \nabla u\|_{2} \leq C(E, Q) \|g\|_\infty,
\]
and similarly for $\|u\|_{\infty}$.

It remains to prove  \eqref{eq:L^inftyFluid}.
Let $R > 0$. Then, 
\begin{align}
	\rho = \int_{\IR^3} g \dd v \leq \int_{\{ |v| < R\}} g \dd v + R^{-2} \int_{\{|v|>R\}} |v|^2 g \dd v
		 \leq C R^3 \|g\|_{\infty}  + C R^{-2} \int_{\{|v|>R\}} |v|^2 g \dd v.
\end{align}
We choose 
\[
	R =  \left(\int_{\IR^3} |v|^2 f \dd v\right)^{1/5} \|g\|_{\infty}^{-1/5}.
\]
Thus,
\[
	\rho \leq \|g\|_{\infty}^{2/5} \left(\int_{\IR^3} |v|^2 g \dd v\right)^{3/5},
\]
and therefore,
\begin{equation}
	\label{eq:rho5over3}
	\|\rho\|_{5/3} \leq \|g\|_{\infty}^{2/5} E^{\frac{3}{5}}.
\end{equation}
Moreover, by definition of $Q$, \eqref{eq:rho5over3} implies for all $1 \leq p \leq 5/3$,
\begin{equation}
	\label{eq:j5over3}
	\|j\|_{p} \leq Q \|\rho \|_{p} \leq C(\|g\|_{\infty},\|g\|_{1},E) Q.
\end{equation}

Sobolev and H\"older's inequality imply
\[
	\| u \|_{10} \leq C \| \nabla^2 u\|_{30/23} \leq  C\|\rho\|_{5/3} \|u\|_{6} + C\|j\|_{30/23} \leq C(\|g\|_{\infty},\|g\|_{1},E) (1 + Q),
\]
where we used \eqref{eq:uInL^6}, \eqref{eq:rho5over3}, and \eqref{eq:j5over3}.
Now, we can repeat the argument, using this improved estimate for $u$ in \eqref{eq:ellipticRegularity}.
This yields
\[
	\|u\|_{30} \leq C(\|g\|_{\infty},\|g\|_{1},E) (1 + Q).
\]
Using again \eqref{eq:ellipticRegularity} yields
\[
	\| \nabla^2 u\|_{30/19} \leq  C(\|g\|_{\infty},\|g\|_{1},E) (1 + Q).
\]
As $30/19 > 3/2$, we can apply Sobolev embedding to get
\[
	\|u\|_{\infty} \leq C\| \nabla^2 u\|_{30/19} + C \|u\|_{6} \leq C(\|g\|_{\infty},\|g\|_{1},E) (1 + Q),
\]
which finishes the proof of \eqref{eq:L^inftyFluid}.
\end{proof}

\subsection{A priori estimates for the particle density}
\label{sec:energyParticles}

\begin{lemma}
	\label{lem:aPrioriWellPosedness}
	Let  $T > 0$ and $f_0 \in W^{1,\infty}(\IR^3 \times \IR^3)$ and 
let $Q_0 > 0$ be minimal such that $\supp f_0 \subset B_{Q_0}(0)$.
	Assume $f \in W^{1,\infty}((0,T) \times \IR^3 \times \IR^3)$ is a solution to \eqref{eq:Vlasov Stokes}
	with $u \in L^\infty((0,T);W^{1,\infty}(\IR^3))$. Then, $f$ is compactly supported on $[0,T]\times\IR^3 \times \IR^3$.
	Let $Q(t)$ be minimal such that $\supp f(t,\cdot,\cdot) \subset B_{Q(t)}(0)$. Furthermore, define
	\begin{align}
		E(t) &:=  \int_{\IR^3 \times \IR^3} \hspace{-1em} |v|^2 f \dd x \dd v.
	\end{align}
	Then,
	\begin{align}
		\label{eq:fL^Infty}
		\|f(t,\cdot,\cdot)\|_{L^\infty(\IR^3\times\IR^3)} &= e^{3 \lambda t}, \\
		\label{eq:massConservation} 
		\| \rho \|_1 &= 1, \\
		\label{eq:kineticEnergyParticles} 
		\partial_t E &= 2 \lambda \left( g \cdot  \int_{\IR^3} j \dd x 
	-   \int_{\IR^3 \times \IR^3} \hspace{-1em}  (u - v)^2 f \dd x \dd v - \|\nabla u \|^2_{L^2(\IR^3)} \right) \\
		\label{eq:kineticEnergyParticlesEstimate}
		&\leq  2 \lambda \bigg( C E^{\frac{1}{2}}	-  \int_{\IR^3 \times \IR^3} \hspace{-1em}  (v - \bar{V})^2 f \dd x \dd v - 
		\|u-\bar{V}\|_{L^2_\rho(\IR^3)}^2 - \|\nabla u \|^2_{L^2(\IR^3)} \bigg)\!, \hspace{1.2em}\\
		\label{eq:particleEnergyEstimate}
		E(t) & \leq C(1+ (\lambda t)^2), \\
		\label{eq:growthOfQ}
		Q(t) & \leq C(t,\lambda).
	\end{align}
\end{lemma}

\begin{proof}
By the regularity assumptions on $f$ and $u$, the characteristics in \eqref{eq:characteristics} are well defined and \eqref{eq:fByCharacteristics} 
holds. This shows that the support of $f$ remains uniformly bounded on compact time intervals.

The  exponential growth of the $L^\infty$-norm of $f$ \eqref{eq:fL^Infty}  follows from  the characteristic equations as we have seen in \eqref{eq:fByCharacteristics}.

Mass conservation \eqref{eq:massConservation} follows directly from integrating the Vlasov equation \eqref{eq:Vlasov Stokes}.

We multiply the Vlasov equation by $|v|^2$ and integrate to find
\begin{equation}
\begin{aligned}
	\partial_t E
	& = 2 \int_{\IR^3 \times \IR^3} \hspace{-1em}  v \cdot \lambda (g + u - v) f \dd x \dd v  \\
	& = 2 \lambda \left( g \cdot \int_{\IR^3 \times \IR^3} \hspace{-1em}  v  f \dd x \dd v
	- \int_{\IR^3 \times \IR^3} \hspace{-1em}    (u - v)^2 f \dd x \dd v 
	+  \int_{\IR^3 \times \IR^3} \hspace{-1em}  u \cdot (u-v) f \dd x \dd v  \right)\\
	& = 2 \lambda \left( g \cdot \int_{\IR^3 \times \IR^3} j \dd x 
	-  \int_{\IR^3 \times \IR^3} \hspace{-1em}    (u - v)^2 f \dd x \dd v - \|\nabla u \|^2_{L^2(\IR^3)} \right).
\end{aligned}
\end{equation}
This yields the identity \eqref{eq:kineticEnergyParticles}. By the Cauchy-Schwarz inequality
\begin{equation}
	\label{eq:jE^1/2}
	\int_{\IR^3} |j| \dd x \leq \int_{\IR^3 \times \IR^3} \hspace{-1em} |v| f \dd v \dd x \leq \|\rho\|_{L^1(\IR^3)}^{1/2} E^{1/2}.
\end{equation}
Moreover, by definition of $\bar{V}$ in \eqref{eq:defJVbar}
\begin{equation}
	\label{eq:Variance}
\begin{aligned}
	\int_{\IR^3 \times \IR^3} \hspace{-1em}    (u - v)^2 f \dd x \dd v   
	&=\int_{\IR^3 \times \IR^3} \!\!  \left( (v - \bar{V})^2 + (\bar{V} - u)^2 - 2 (v - \bar{V})(\bar{V} - u) \right)f \dd x \dd v \\
	&= \int_{\IR^3 \times \IR^3} \hspace{-1em}  (v - \bar{V})^2 f \dd x \dd v + \|u-\bar{V}\|_{L^2_\rho(\IR^3)}^2.
\end{aligned}
\end{equation}
Using \eqref{eq:jE^1/2} and \eqref{eq:Variance} shows \eqref{eq:kineticEnergyParticlesEstimate}.

In particular
\[
	\partial_t E \leq C \lambda E^{1/2}.
\]
This proves \eqref{eq:particleEnergyEstimate} by a comparison principle for ODEs.

The characteristic equation for $V$ in \eqref{eq:characteristics} implies 
\begin{align*}
	|V(t,0,x,v)|& = \left|e^{-\lambda t} \left( v + \lambda \int_0^t  e^{\lambda s} (g + u(s,X(s,0,x,v))) \dd s \right)\right| \\
	&\leq e^{-\lambda t} v + |g| + \int_0^t \|u(s\cdot)\|_{L^\infty(\IR^3)} \dd s.
\end{align*}
Thus, for all $ (x,v) \in \supp f_0$, we get by Lemma \ref{lem:WellPosednessFluid}, \eqref{eq:fL^Infty}, \eqref{eq:massConservation}, and
\eqref{eq:particleEnergyEstimate}
\begin{equation}
\begin{aligned}
	\label{eq:velocityBound}
	|V(t,0,x,v)| & \leq Q_0 + 1 + 
	C(\|f\|_{L^\infty((0,t) \times \IR^3\times \IR^3)},\|E\|_{L^\infty(0,t)}) \int_0^t (1 + Q(s)) \dd s \\
	& \leq  C + C (\lambda t) \int_0^t (1+Q(s)) \dd s.
\end{aligned}
\end{equation}
By the equation for $X$, we get for all $ (x,v) \in \supp f_0$
\begin{equation}
	\label{eq:PositionBound}
	|X(t,0,x,v)| \leq Q_0 + \int_0^t |V(s,0,x,v)| \dd s \leq Q_0 + t C (\lambda t) \int_0^t (1+Q(s)) \dd s.
\end{equation}
Hence,
\[
	Q(t) \leq  \sup_{(x,v) \in \supp f_0} |( X(t,0,x,v),V(t,0,x,v) )| \leq  C +  (1+t) C(\lambda t) \int_0^t (1+Q(s)) \dd s.
\]
Gronwall's equation yields \eqref{eq:growthOfQ}.
\end{proof}

\subsection{Well-posedness by the Banach fixed point theorem}

\begin{proposition}
	\label{pro:WellPosedness}
	Let $f_0 \in W^{1,\infty}(\IR^3 \times \IR^3) $  with compact support.
	Then, for all $T>0$, there exists a unique solution  $f \in W^{1,\infty}((0,T) \times \IR^3 \times \IR^3)$ to
	 \eqref{eq:Vlasov Stokes} with $u \in L^{\infty}((0,T);W^{2,\infty}(\IR^3)) \cap W^{1,\infty}((0,T)\times\IR^3))$. 
\end{proposition}

\begin{proof}
We want to prove existence of solutions using the Banach fixed point theorem.
Let $Q_1,E_1 > 0$. We define the metric space, where we want to prove contractiveness,
\begin{align}
	Y := \bigg\{ h \in L^{\infty}((0,T)\times \IR^3 \times \IR^3) \colon &h \geq 0, \|h(t,\cdot)\|_{L^1(\IR^3)} = \|f_0\|_{L^1(\IR^3)}, \\ 
	 &\int_{\IR^3\times\IR^3} \hspace{-1em}  (1 + |v|^2) h \dd x \dd v \leq E_1, \supp h \subset [0,T] \times \overline{B_{Q_1}(0)} \bigg\}.
\end{align}
Then, $Y$ is a complete metric space.
Let $T > 0$ and $h_1, h_2 \in Y$. For $i=1,2$, we define $u_i$ to be the solution to
\[
	- \Delta u_i + \nabla p = \int_0^\infty \int_{\IR^3} (v-u_i) h_i \dd v.
\]

We define the characteristics $(X_i,V_i)(s,t,x,v)$ analogously to \eqref{eq:characteristics} by
\begin{align}
	\partial_s (X_i,V_i)(s,t,x,v) &= (V_i(s,t,x,v),g + u_i(s,X_i(s,t,x,v)) - V_i(s,t,x,v)), \\
	(X_i,V_i)(t,t,x,v) = (x,v).
\end{align}
Then, the solutions to the equation
\begin{equation}
	\partial_t f_i + v \cdot \nabla_x f_i + \lambda \dv_v \left(  g f_i +  (u_i-v) f_i  \right) = 0,
\end{equation}
with initial datum $f_0$ is given by
\begin{equation}
	\label{eq:fByCharacteristicsContraction}
	f_i(t,x,v) = e^{3\lambda t} f_0 ((X_i,V_i)(0,t,x,v)),
\end{equation}
and $f_i \in W^{1,\infty}((0,T)\times\IR^3\times\IR^3)$.
We estimate
\begin{equation}
	\label{eq:difff1f2}
	|f_1(t,x,v) - f_2(t,x,v)| 
	\leq e^{3\lambda t} \| \nabla f_0 \|_{L^\infty(\IR^3\times\IR^3)} |(X_1,V_1)(0,t,x,v) - (X_2,V_2)(0,t,x,v)|.
\end{equation}
Furthermore, writing $X_i(s)$ instead of $X_i(s,t,x,v)$ and similar for $V_i$, we have for all $0 \leq s \leq t$
\begin{align}
	&|(X_1,V_1)(s) - (X_2,V_2)(s)| \\
	&\leq \int_s^t |\left(V_1(\tau) - V_2(\tau), \lambda \left(u_1(\tau,X_1(\tau))- u_2(\tau,X_2(\tau)) - V_1(\tau) + V_2(\tau) \right)\right)| \dd \tau \\
	& \leq  \int_s^t |V_1(\tau) - V_2(\tau)| + \|\nabla u_1(\tau,\cdot)\|_{L^\infty(\IR^3)} |X_1(\tau) - X_2(\tau)| 
	+ \| u_1(\tau,\cdot) - u_2(\tau,\cdot)\|_{L^\infty(\IR^3)} \dd \tau  \\
	& \leq  C(X,Q_1,E_1) V
		\int_s^t |(X_1,V_1)(\tau) - (X_2,V_2)(\tau)| \dd \tau  +  C(Q_1,E_1) (t-s) \|g_1 - g_2 \|_{L^\infty(\IR^3)},
\end{align}
where we used Lemma \ref{lem:WellPosednessFluid}.
Gronwall's inequality implies
\[
	|(X_1,V_1)(t) - (X_2,V_2)(t)| \leq C(Q_1,E_1) t \|g_1 - g_2 \|_{L^\infty(\IR^3)} 
	\exp\left(C(Q_1,E_1) \|g_1\|_{L^\infty(\IR^3)}t \right).
\]
Inserting this in \eqref{eq:difff1f2} yields
\begin{equation}
	\label{eq:contraction}
	\begin{aligned}
	&\|f_1 - f_2\|_{L^\infty((0,T) \times \IR^3 \times \IR^3)} \\
	&\leq T e^{3T } C(Q_1,E_1) \| \nabla f_0 \|_{L^\infty(\IR^3)} \|g_1 - g_2 \|_{L^\infty(\IR^3)}
	\exp\left(C(Q_1,E_1) T \|g_1\|_{L^\infty(\IR^3)} \right)
	\end{aligned}
\end{equation}

For $L>0$, consider $B_L(0) \subset Y$. 
Then, for all $L$, equation \eqref{eq:contraction} implies that there exists $T>0$ such that the mapping
$g \mapsto f$ is contractive. We have to check that $g \in B_L(0)$ implies $f \in B_L(0)$.
First, 
\begin{equation}
	\|f(t,\cdot,\cdot)\|_{L^1(\IR^3)} = \|f_0\|_{L^1(\IR^3)} \label{eq:massConservationContraction}
\end{equation}
follows from the equation.
Moreover, for any $L > \|f_0\|_{L^\infty( \IR^3 \times \IR^3)}$, equation \eqref{eq:fByCharacteristicsContraction} implies that we can choose $T$ sufficiently small such that
\begin{equation}
	\label{eq:Linftyf}
	\|f\|_{L^\infty((0,T) \times \IR^3 \times \IR^3)} = \|f_0\|_{L^\infty(\IR^3 \times \IR^3)} e^{3 \lambda T} \leq L.
\end{equation}

Furthermore, we have
\begin{align}
	\partial_t \int_{\IR^3} \int_{\IR^3} |v|^2 f \dd x \dd v 
	& = 2 \int_{\IR^3} \int_{\IR^3} v \cdot (g + u - v) f \dd x \dd v  \\
	& \leq   2(|g| +  \|u\|_{L^\infty(\IR^3)}) \int_{\IR^3} \int_{\IR^3} (1 + |v|^2)  f \dd x \dd v.
\end{align}
Hence, using mass conservation, equation \eqref{eq:massConservationContraction},
\[
	\partial_t \int_{\IR^3\times\IR^3} \hspace{-1em} (1 + |v|^2) f \dd x \dd v \leq 
	 (|g|+\|u\|_{L^\infty(\IR^3)}) \int_{\IR^3\times\IR^3} \hspace{-1em}(1 + |v|^2) f \dd x \dd v.
\]
Therefore, Lemma \ref{lem:WellPosednessFluid} and  Gronwall's inequality imply
\begin{equation}
	\int_{\IR^3\times\IR^3} \hspace{-1em} (1 + |v|^2) f \dd x \dd v 
	\leq \int_{\IR^3\times\IR^3} \hspace{-1em} (1 + |v|^2) f_0 \dd v \dd x \exp (C(Q_1,E_1)L t).
\end{equation}
Thus, for any $E_1 > \int_{\IR^3\times\IR^3}  (1 + |v|^2) f_0 \dd v \dd x$, we can choose
$T$ small enough such that $	\int_{\IR^3\times\IR^3} (1 + |v|^2) f \dd x \dd v \leq E_1$
for all $t \leq T$.

Finally, we need to control the support of $f$. To do this, we follow the same argument as in the last part of the proof of
 Lemma \ref{lem:aPrioriWellPosedness} to get
 \[
 		Q(t) \leq  Q_0 +  (1+t)  \int_0^t C(L,E_1,Q_1) \dd s \leq  Q_0 +  (1+t) t C(L,E_1,Q_1).
 \]
 Again, for any $Q_1 > Q_0$, we can choose
$T$ small enough such that $	Q(t) \leq Q_1$
for all $t \leq T$.

Therefore, by the Banach fixed point theorem, we get local in time existence of solutions to \eqref{eq:Vlasov Stokes}.
Global existence follows directly from the a priori estimates in Lemma \ref{lem:aPrioriWellPosedness}, since these ensure that
all the relevant quantities for the fixed point argument do not blow up in finite time.

Since $f \in W^{1,\infty}((0,T)\times\IR^3\times\IR^3)$ with uniform compact support, higher regularity of $u$ follows from taking derivatives
in the Brinkman equations in \eqref{eq:Vlasov Stokes} and using regularity theory for Stokes equations similar as in the proof of
Lemma \ref{lem:WellPosednessFluid}.
\end{proof}

%% file: A_priori_estimates.tex
\section{Uniform estimates on $\bm{\rho_\lambda}$ and $\bm{u_\lambda}$}
In the following, we assume that $(f,u)$ is the solution to the Vlasov-Stokes equations \eqref{eq:Vlasov Stokes} 
for some $\lambda > 0$ and some compactly supported initial datum $f_0 \in W^{1,\infty}(\IR^3\times\IR^3)$.
In this section we want to derive a priori estimates for these solutions that do not depend on $\lambda$.
This is why we cannot use the a priori estimates derived in Lemma \ref{lem:aPrioriWellPosedness}.
However, the drawback of the estimates that we prove in this section is that they allow for blow-up in finite time.
This is also why they are not suitable in the proof of global well-posedness, that we showed in the previous section.
Later, we will use the limit equation in order to show that the estimates derived here allow for uniform estimates for arbitrary times.

Again, we denote by $C$ any constant, which only depends on $f_0$ and may change from line to line.

\subsection{Estimates for the fluid velocity}

%

In this subsection we show that the fluid velocity as well as the particle velocity is controlled by 
$\|\rho\|_{L^\infty(\IR^3)}$, uniformly in $\lambda$, which means that high velocities can only occur if particles concentrate in position space.
This also implies control on the particle  positions and velocities

The proof is based on the energy identity from Lemma \ref{lem:aPrioriWellPosedness}, equation \eqref{eq:kineticEnergyParticles}, and the 
subsequent estimate \eqref{eq:kineticEnergyParticlesEstimate}.
The idea is to estimate the sum of the quadratic terms in that expression, which have a negative sign, by $E(t)$ from below. 
The following Lemma, which is a general observation on weighted $L^2$-spaces, shows why such an estimate is true  if $\|\rho\|_{L^{3/2}(\IR^3)}$ is not too large.

Having shown this estimate, the quadratic terms in \eqref{eq:kineticEnergyParticlesEstimate} dominate the linear term, which has been estimated by $E(t)^{1/2}$. This leads to control of $E$ uniformly in $\lambda$.
\begin{lemma}
	\label{lem:frictionCoercive}	
	There exists a constant $c_0$, such that for all nonnegative $\sigma \in L^{3/2}(\IR^3)$,  $h \in L^2(\sigma)$ and $w \in H^1(\IR^3)$, 
	\[
		\| \nabla w \|_{L^2(\IR^3)}^2 + \| w - h \|_{L^2_\sigma(\IR^3)}^2 \geq c_0 \min \{ \|\sigma\|_{L^{3/2}(\IR^3)}^{-1}, 1\} \|h\|_{L^2_\sigma(\IR^3)}^2.
	\]
\end{lemma}

\begin{proof}
We estimate using the critical Sobolev inequality
\begin{equation}
	\label{eq:frictionCoerciveSobolev}
	\| w \|^2_{L^2_\sigma(\IR^3)} \leq \|w\|_{L^6(\IR^3)}^2 \|\sigma\|_{L^{3/2}(\IR^3)} \leq C \| \nabla w \|_{L^2(\IR^3)}^2  \|\sigma\|_{L^{3/2}(\IR^3)}.
\end{equation}
We have for any $ \theta > 0 $ and any $a,b \in H$ for some Hilbert space $H$
	\[
		\|a - b\|^2 = \| a \|^2 + \|b\|^2 - 2(a,b) \geq (1 - \theta) \|a\|^2 + (1- \frac{1}{\theta}) \|b\|^2.
	\]
Applying this with $1 - \theta := - C^{-1} \|\sigma \|_{L^{3/2}(\IR^3)}^{-1}$, where $C$ is the constant from equation \eqref{eq:frictionCoerciveSobolev},
we find
\[
	\| \nabla w \|_{L^2(\IR^3)}^2 + \| w - h \|_{L^2_\sigma(\IR^3)}^2 \geq \frac{\theta - 1}{\theta} \|h\|_{L^2_\sigma(\IR^3)}^2.
\]
To conclude, we notice that
\[
	\frac{\theta - 1}{\theta} = \frac{C^{-1} \|\sigma \|_{L^{3/2}(\IR^3)}^{-1}}{1+C^{-1} \|\sigma \|_{L^{3/2}(\IR^3)}^{-1}}
	\geq c_0 \min \{ \|\sigma\|_{L^{3/2}(\IR^3)}^{-1}, 1\}. \qedhere
\]
\end{proof}

\begin{lemma}
	\label{lem:boundsU}
	There exists a constant $C$ that depends only on $f_0$  such that for all $\lambda >0$ and all $t > 0$, we have
	\begin{align}
		\label{eq:energyBound}
		E(t) & \leq C\sup_{s\leq t}\|\rho\|_{L^{\infty}(\IR^3)}^{\frac{2}{3}}, \\
		\label{eq:uInfty}
		\| u(t,\cdot)\|_{L^{\infty}(\IR^3)} &\leq C \sup_{s\leq t} \|\rho\|_{L^\infty(\IR^3)}, \\
		\label{eq:nablaU}
		\|\nabla u(t,\cdot)\|_{L^{\infty}(\IR^3)} &\leq C \sup_{s\leq t}  \|\rho\|_{L^\infty(\IR^3)}^{2}, \\
		\label{eq:VBarInfty}
		\| \bar{V}(t,\cdot)\|_{L^{\infty}(\IR^3)} &\leq C \sup_{s\leq t} \|\rho\|_{L^\infty(\IR^3)},
	\end{align}
	where $\bar{V}$ is the average particle velocity defined in \eqref{eq:defJVbar}.
	
	Moreover, for all $(x,v) \in \supp f_0$,
	\begin{align}
		\label{eq:boundOnV}
			|V(t,0,x,v)| &\leq C \sup_{s\leq t}  \|\rho\|_{L^\infty(\IR^3)}, \\
			|X(t,0,x,v)| &\leq C t  \sup_{s\leq t}  \|\rho\|_{L^\infty(\IR^3)}.
		\label{eq:boundOnX}
	\end{align}
\end{lemma}

\begin{proof}
	By the energy estimate \eqref{eq:kineticEnergyParticlesEstimate}  from Lemma \ref{lem:aPrioriWellPosedness} and Lemma \ref{lem:frictionCoercive}, 
	we have for the energy of the particles
\begin{align}
	\partial_t E  & \leq  2 \lambda \left( C E^{\frac{1}{2}}	-  \int_{\IR^3 \times \IR^3} \hspace{-1em}  (v - \bar{V})^2 f \dd x \dd v - 
		\|u-\bar{V}\|_{L^2_\rho(\IR^3)}^2 - \|\nabla u \|_{L^2(\IR^3)}) \right) \\
	& \leq 2 \lambda \left( C  E^{\frac{1}{2}}
	-  \int_{\IR^3 \times \IR^3} \hspace{-1em}  (v - \bar{V})^2 f \dd x \dd v -
	 c_0 \min\{ \|\rho\|_{L^{3/2}(\IR^3)}^{-1},1\} \|\bar{V}\|_{L^2_\rho(\IR^3)}^2 \right)\\
	& \leq  2 \lambda  \left(C E^{\frac{1}{2}} -   c_0 \min\{ \|\rho\|_{L^{3/2}(\IR^3)}^{-1},1\} E \right).
\end{align}
A comparison principle for ODEs implies
\begin{equation}
	\label{eq:energyEstimateRho}
	E^{\frac{1}{2}}(t)  \leq E(0)^{\frac{1}{2}} e^{-2\lambda t} + \frac{C}{c_0} \sup_{s\leq t} \max \{\|\rho\|_{L^{3/2}(\IR^3)},1 \} 
	\leq C  \sup_{s\leq t}\|\rho\|_{L^{3/2}(\IR^3)} \leq C\sup_{s\leq t}\|\rho\|_{L^{\infty}(\IR^3)}^{\frac{1}{3}},
\end{equation}
where   we used that the $L^1$-norm of $\rho$ is constant in time by \eqref{eq:massConservation}. Note that here and in the following we also use that $C$ might depend on $f_0$ in order to get rid of lower order terms (using that if $f_0 = 0$, the solution $f$ is also trivial).
This proves \eqref{eq:energyBound}.

Recall from \eqref{eq:energyEstimateFluid} that $\|\bar V\|_{L^2_\rho(\IR^3)} \leq E^{\frac{1}{2}}$.
Thus, \eqref{eq:energyEstimateRho} yields
\begin{equation}
	\label{eq:barVRho}
	\|\bar V(t)\|_{L^2_\rho(\IR^3)} \leq E^{\frac{1}{2}}(t) \leq  C  \sup_{s\leq t}\|\rho\|_{L^{\infty}(\IR^3)}^{\frac{1}{3}}.
\end{equation}
Using regularity theory for the Stokes equations (see \cite{Ga11}) together with \eqref{eq:energyEstimateFluid} and \eqref{eq:barVRho} yields
\begin{align}
	\|\nabla^2 u \|_{L^2(\IR^3)} &\leq C\|\rho u \|_{L^2(\IR^3)} + C\|\rho \bar V\|_{L^2(\IR^3)} 
	\leq C\|u\|_{L^6(\IR^3)} \|\rho\|_{L^3(\IR^3)} + C\|\rho \bar V\|_{L^2(\IR^3)} \\
	 &\leq C \|\bar V\|_{L^2_\rho(\IR^3)} \|\rho\|_{L^{\infty}(\IR^3)}^{\frac{2}{3}} 
	 +  \|\bar V\|_{L^2_\rho(\IR^3)} \|\rho\|_{L^{\infty}(\IR^3)}^{\frac{1}{2}} \\
	 &\leq  C \|\bar V\|_{L^2_\rho(\IR^3)} \|\rho\|_{L^{\infty}(\IR^3)}^{\frac{2}{3}} \leq C \sup_{s\leq t} \|\rho\|_{L^\infty(\IR^3)}.
\end{align}
Sobolev inequality and \eqref{eq:energyEstimateFluid} yield
\[
	\|u\|_{L^\infty(\IR^3)} \leq \|u\|_{C^{0,\frac{1}{2}}(\IR^3)} \leq C \|u\|_{W^{1,6}(\IR^3)} 
	\leq C \|\nabla u \|_{W^{1,2}(\IR^3)} \leq C \sup_{s\leq t} \|\rho\|_{L^\infty(\IR^3)}.
\]
This proves \eqref{eq:uInfty}.

Using the characteristic equations \eqref{eq:characteristics}, we find for all $(x,v) \in \supp f_0$
\begin{equation}
	|V(t,0,x,v)| \leq e^{-\lambda t} |v| + |g| + \sup_{s\leq t} \|u(t)\|_{L^\infty(\IR^3)} 
	\leq  C \sup_{s\leq t} \|\rho\|_{L^\infty(\IR^3)},
\end{equation}
with a constant $C$ that depends on the support of $f_0$.
This proves \eqref{eq:boundOnV}. Moreover, using the equation for $X$, \eqref{eq:boundOnV} implies \eqref{eq:boundOnX}.

Furthermore, by \eqref{eq:boundOnV}
\[
	\|\bar{V}(t)\|_{L^\infty(\IR^3)} \leq C \sup_{s\leq t} \|\rho\|_{L^\infty(\IR^3)},
\]
which proves \eqref{eq:VBarInfty}.
This can be used again to derive a bound for $\nabla^2 u$ in $L^p(\IR^3)$ to get \eqref{eq:nablaU}.
More precisely, 
\begin{align}
	\|\nabla^2 u \|_{L^6(\IR^3)} &\leq \|u\|_{L^6(\IR^3)} \|\rho\|_{L^\infty(\IR^3)} + \|\rho \bar V\|_{L^6(\IR^3)} \\
	 &\leq C \|\bar V\|_{L^2_\rho(\IR^3)} \|\rho\|_{L^{\infty}(\IR^3)}  
	 + \|\bar V\|^{\frac{1}{3}}_{L^2_\rho(\IR^3)} \|\bar V\|_{L^\infty(\IR^3)}^{\frac{2}{3}} \|\rho\|_{L^\infty}^{\frac{5}{6}}
	 \leq C \sup_{s\leq t} \|\rho\|_{L^\infty(\IR^3)}^2.
\end{align}
Thus,
\[
	\|\nabla u\|_{L^\infty(\IR^3)} \leq C \sup_{s\leq t} \|\rho\|_{L^\infty(\IR^3)}^2. \qedhere
\]
\end{proof}

\subsection{Estimates for the particle density}
In this subsection we prove estimates on $\rho$ that are uniform in $\lambda$ for sufficiently $\lambda$ sufficiently large but depend on $u$.
Then, we will combine these estimates with the ones from Lemma \ref{lem:boundsU} in order to get estimates on $\rho$ independent of $\lambda$
and $u$ but only for small times.

We first prove a small lemma on estimates for ODEs that will be used several times analyzing the characteristics.
\begin{lemma}
	\label{lem:ODEEstimates}
	Let $T>0$ and $a,b:[0,T] \to \IR_+$ be Lipschitz continuous. Let $\alpha \colon [0,T] \to \IR_+$ be continuous and 
	$\lambda \geq 4 \max\{ 1,\|\alpha\|_{L^\infty(0,T)} \}$. 
	Let $\beta \geq 0$ be some constant and assume that on $(0,T)$
		\begin{align}
			|\dot{a}| &\leq b, \\
			\dot{b} &\leq \lambda(\alpha a - b) + \beta e^{-\lambda s}.
		\end{align}
	\begin{enumerate}[label=(\roman*)]
		
	\item	\label{it:ODEEstimates1}
		If $a(T) = 0$, then for all $s,t \in [0,T]$ with $s \leq t$
		\begin{align}
			a(t) &\leq \frac{2}{\lambda} b(t) + \frac{4}{\lambda^2} \beta e^{-\lambda t}, \label{eq:ODEEstimates1a} \\
			b(t) & \leq  \exp \left(\int_s^t -\lambda + 2 \alpha(\tau) \dd \tau \right) 
						\left( b(s) + \frac{2 \beta}{\lambda} e^{-\lambda s} \right). \label{eq:ODEEstimates1b}
		\end{align}

	\item  \label{it:ODEEstimates2}
	If $\beta = 0$ and $b(0) = 0$, then for all $t \in [0,T]$
		\[
			b(t) \leq  2 \|\alpha\|_{L^\infty(0,T)} a.
		\]
	\end{enumerate}
\end{lemma}

\begin{proof}
	We define
	\[
		z(s) := b(s) - \frac{\lambda}{2} a(s) + \frac{2}{\lambda} \beta e^{-\lambda s}.
	\]
	Then, if $a(T) = 0$,
	\[
		z(T) = b(T) + \frac{2}{\lambda} \beta e^{-\lambda T} \geq 0,
	\]
	and
	\[
		\dot{z} \leq \lambda \left( \alpha a - \frac{b}{2} \right)  - \beta e^{- \lambda s} 
		= \lambda \left( \alpha a - \frac{\lambda}{4} a - \frac{z}{2} + \frac{\beta}{\lambda} e^{- \lambda s} \right)  + \beta e^{- \lambda s}
		\leq - \frac{\lambda}{2} z.
	\]
	Hence, (applying Gronwall's inequality to $-z(T-t)$) we find $z \geq 0$ in $[0,T]$. This proves \eqref{eq:ODEEstimates1a}.
	Moreover, \eqref{eq:ODEEstimates1a} implies
	\[
		\dot{b} \leq (2 \alpha - \lambda) b + \left( 1+ \frac{4}{\lambda} \right) \beta  e^{-\lambda s} 
		\leq  (2 \alpha - \lambda) b + 2 \beta  e^{-\lambda s}.
	\]
	Thus, using the comparison principle for ODEs yields \eqref{eq:ODEEstimates1b}.
	
	In order to prove \ref{it:ODEEstimates2}, we define $z := 2 \|\alpha\|_{L^\infty(0,T)} a - b$. Then, $b(0) = 0$ implies $z(0) \geq 0$.
	Using the equations for $a$ and $b$, one obtains $\dot{z} \geq - (\lambda/2) z$ as in the proof of part (i). This implies $z \geq 0$ in $[0,T]$,
	and the assertion follows.
\end{proof}

Using the previous Lemma, we are able to prove that the particle velocities concentrate in regions of size $e^{-\lambda t}$ with an error due to fluctuations of the fluid velocity. Based on this result and equation \eqref{eq:fByCharacteristics}, we also prove an estimate for $\rho$.

\begin{lemma}
	\label{lem:biLipschitzV}
	Let $T > 0$ and assume $\lambda \geq 4(1+ \|\nabla u\|_{L^\infty((0,T) \times \IR^3)})$. Then, for all $t < T$ and all $x \in \IR^3$, the map
	\[
		v \mapsto V(0,t,x,v)
	\]
	is bi-Lipschitz. In particular its inverse $W(t,x,w)$ is well defined, and
	\begin{equation}
		\label{eq:rhoByW}
		\rho(t,x) = \int_{\IR^3} e^{ 3 \lambda t} f_0(X(0,t,x,W(t,x,w)),w) \det \nabla_w W(t,x,w)  \dd w.
	\end{equation}
	Moreover, denoting
	\begin{equation}
		\label{eq:defM}
		M(t) := \exp \left(\int_0^t   2 \|\nabla u(s,\cdot) \|_{L^\infty(\IR^3)} \dd s \right),
	\end{equation}
	we have
	\begin{align}
		|\nabla_v V(0,t,x,v)| &\leq M(t) e^{\lambda t}, \label{eq:LipschitzV} \\
		|\nabla_w W(t,x,w)| &\leq M(t) e^{-\lambda t}, \label{eq:LipschitzW} \\
		0 \leq \det \nabla_w W(t,x,w) &\leq M(t)^3 e^{-3\lambda t}. \label{eq:JacobianW}
	\end{align}
	Furthermore,
	\begin{equation}
		\label{eq:estimateRho}
		\| \rho(t,\cdot)\|_{L^\infty(\IR^3)} \leq C_0 M(t)^3,
	\end{equation}
	where the constant depends only on $f_0$.
\end{lemma}

\begin{proof}
	We fix $t$, $x$, $v_1$, and $v_2$ and define
\begin{align}
	a(s) &= | X(s,t,x,v_1) - X(s,t,x,v_2) |,\\
	b(s) &= | V(s,t,x,v_1) - V(s,t,x,v_2)|.
\end{align}
Then, 
\begin{align}
	|\dot{a}| & \leq b, \qquad &&a(t) = 0, \\
	\dot{b} &\leq \lambda ( \|\nabla u(s,\cdot) \|_{L^\infty(\IR^3)} a - b), \qquad &&b(t) = |v_1 - v_2|.
\end{align}
Then, with $\alpha(s) :=  \|\nabla u(s,\cdot) \|_{L^\infty(\IR^3)}$ and $\beta = 0$, we can apply 
Lemma \ref{lem:ODEEstimates}\ref{it:ODEEstimates1} to deduce
\[
	b(t) \leq b(0) M(t) e^{-\lambda t},
\]
which implies
\begin{equation}
	\label{eq:LipschitzInverse}
	b(0) \geq M(t)^{-1} e^{\lambda t} |v_1 - v_2|.
\end{equation}

Note that the first inequality in  \eqref{eq:ODEEstimates1a} also implies 
\[	
	a(t) \leq \frac{2}{\lambda} b(t).
\]
Hence,
\[
	\dot{b} \geq\lambda ( - \|\nabla u(s,\cdot) \|_{L^\infty(\IR^3)} a - b) \geq \left( - \lambda   - 2\|\nabla u(s,\cdot) \|_{L^\infty(\IR^3)} \right) b.
\]
Thus
\begin{equation}
	\label{eq:Lipschitz}
	b(0) \leq e^{\lambda t} M(t) |v_1 - v_2|.
\end{equation}

Estimates \eqref{eq:LipschitzInverse} and \eqref{eq:Lipschitz} imply that the map $v \mapsto V(0,t,x,v)$ is bi-Lipschitz
and yield the bounds \eqref{eq:LipschitzV}, \eqref{eq:LipschitzW}, and \eqref{eq:JacobianW}. The Jacobian of $W$ is positive since $W(0,x,v) = w$ and 
the Jacobian is continuous in $t$, which follows from the definition of $V$ and regularity of $u$ proven in Proposition \ref{pro:WellPosedness}.

Moreover, recalling \eqref{eq:fByCharacteristics}, these estimates imply
\begin{align}
	\rho(t,x) = \int_{\IR^3} f(t,x,v) \dd v& = \int_{\IR^3} e^{ 3 \lambda t} f_0(X(0,t,x,v),V(0,t,x,v))  \dd v \\
	&= \int_{\IR^3} e^{ 3 \lambda t} f_0(X(0,t,x,W(t,x,w)),w) \det \nabla_w W(t,x,w)  \dd w \\
	&\leq \int_{\IR^3} M(t)^3 f_0(X(0,t,x,W(t,x,w)),w) \dd w \\
	& \leq C_0 M(t)^3,
\end{align}
which finishes the proof.
\end{proof}

%

We define 
\begin{equation}
	\label{eq:defT_ast}
	T_\ast := \sup \Big\{ t \geq 0 \colon \limsup_{\lambda \to \infty} \|\rho_\lambda\|_{L^\infty((0,t)\times\IR^3)} < \infty \Big\}
\end{equation}
In the lemma below, we prove that $T_\ast > 0$. Later we will show the convergence to the limit equation \eqref{eq:limitEquation}
first  only up to times $ T < T_\ast$ and finally, we will show $T_\ast = \infty$ using the convergence result for times $T < T_\ast$.

\begin{lemma}
	\label{lem:T_astPositive}
	Let $T_\ast$ be defined as in \eqref{eq:defT_ast}. Then,
	\[
		T_\ast >0.
	\]
\end{lemma}

\begin{proof}
By Lemma \ref{lem:boundsU}, we have for all $t > 0$
\[
	\|\nabla u_\lambda \|_{L^\infty((0,t) \times \IR^3)} \leq C \sup_{s \leq t} \|\rho_\lambda(s,\cdot)\|^2_{L^\infty(\IR^3)}.
\]
Moreover, by Lemma \ref{lem:biLipschitzV}, if $\lambda \geq 4 (\|\nabla u_\lambda \|_{L^\infty(0,t) \times \IR^3)} + 1)$,
then
\[
	\sup_{s \leq t} \|\rho_\lambda(s,\cdot)\|^2_{L^\infty(\IR^3)} 
	\leq C_0 \exp\left(2 \int_0^t \|\nabla u_\lambda (s,\cdot)\|_{L^\infty(\IR^3)} \dd s \right).
\]
Combining these two estimates, we see that $\lambda \geq  C \sup_{s \leq t} \|\rho_\lambda(s,\cdot)\|^2_{L^\infty(\IR^3)}$
implies
\begin{equation}
	\label{eq:growRho}
	\sup_{s \leq t} \|\rho_\lambda(s,\cdot)\|_{L^\infty(\IR^3)} 
	\leq C_0 \exp\left(C t \sup_{s \leq t} \|\rho_\lambda(s,\cdot)\|^2_{L^\infty(\IR^3)}\right).
\end{equation}
We define 
\[
	T_\lambda := \sup \{t \geq 0 \colon \sup_{s \leq t} \|\rho_\lambda(s,\cdot)\|_{L^\infty(\IR^3)} \leq 2 C_0\}.  
\]
Then, $T_\lambda > 0$ as $\rho_\lambda$ is continuous (and $C_0$ is chosen such that $\|\rho(0,\cdot)\|_{L^\infty(\IR^3)} \leq C_0$).
Moreover, \eqref{eq:growRho} implies for all $\lambda \geq 4 (C C_0^2+1)$ and all $t < T_\lambda$
\[
	\sup_{s \leq t} \|\rho_\lambda(s,\cdot)\|_{L^\infty(\IR^3)} \leq C_0 \exp( C C_0^2 t).
\]
As $\rho_\lambda$ is continuous, this yields for all $\lambda \geq  4 (C C_0^2+1)$
\[
	T_\lambda \geq \frac{\log(2)}{C C_0^2},
\]
which is independent of $\lambda$. Thus,
\[
	T_\ast \geq \inf_{\lambda \geq  4 (C C_0^2+1)} T_\lambda > 0. \qedhere
\]
\end{proof}

\subsection{Higher order estimates}

In this subsection, we prove estimates on $\partial_t \rho$ and $\nabla \rho$ which are uniform in $\lambda$ for times $T < T_\ast$.
On the one hand, this yields compactness of $\rho_\lambda$ in H\"older spaces. On the other hand, we will also need these
estimates in order to prove that the functions $\tilde{u}_\lambda$ defined in \eqref{eq:uTilde} are close to $u_\lambda$ for large values of $\lambda$.

From now on, any constant $C$ might depend on $T$ but not on $\lambda$. In particular, for $T<T_\ast$, $C$ might depend on 
$\limsup_{\lambda \to \infty} \|\rho_\lambda\|_{L^\infty((0,T)\times\IR^3)}$.

\begin{lemma}
	\label{lem:APriori}
	Let $T < T_\ast$. Then, there exist  $\lambda_0$ and $C$ depending on $T$ and $f_0$ such that
	for all $\lambda \geq \lambda_0$ and all multiindices $\beta \in \IN^3$,
	\begin{align}
		\|\rho\|_{W^{1,\infty}((0,T)\times\IR^3)} &\leq C, \label{eq:rhoLipschitz}\\
		\|u\|_{L^\infty((0,T_0);W^{2,\infty}(\IR^3))} &\leq C, \label{eq:nabla^2U} \\
		\|\bar V\|_{L^\infty((0,T_0)\times\IR^3)} &\leq C, \label{eq:VBarL^Infty}\\	
	\Big\| \nabla_x \int_{\IR^3} v^\beta f \dd v \Big\|_{L^\infty((0,T_0)\times\IR^3)} & \leq C. \label{eq:nablaRho}\\
	\end{align}
	Moreover, the support of $f$ is uniformly bounded in $\lambda$ up to time $T$.
\end{lemma}

\begin{proof}
By definition of $T_\ast$, there is  some $\lambda_0$ such that for all $\lambda \geq \lambda_0$ 
\begin{align}
	\|\rho\|_{L^\infty((0,T)\times\IR^3)} &\leq C \label{eq:rhoInfty} .
\end{align}
Thus, Lemma \ref{lem:boundsU} yields \eqref{eq:VBarL^Infty} and
\begin{align}
		\|u\|_{L^\infty((0,T_0);W^{1,\infty}(\IR^3))} &\leq C. \label{eq:uInftyT_ast}
\end{align}
Using this, we have $M(t) \leq C$ for all $t \leq T$, where $M$ is the quantity from \eqref{eq:defM} in Lemma \ref{lem:biLipschitzV}.
Moreover, we can assume that $\lambda_0$ has been chosen sucht that for all $\lambda \geq \lambda_0$
\begin{equation}
	\label{eq:lambdaLarge}
	\lambda \geq 4 (1+\|\nabla u \|_{L^\infty((0,T)\times\IR^3)}).
\end{equation}
In the following, we only consider $\lambda \geq \lambda_0$.

By Lemma \ref{lem:biLipschitzV}, $V(0,t,x,v)$ is invertible with inverse $W(t,x,v)$, and  we define
\begin{equation}
	\label{eq:defY}
	\begin{aligned}
	Y(s,t,x,w)&:= X(s,t,x,W(t,x,w)), \\
	U(s,t,x,w) &:= V(s,t,x,W(t,x,w)).
	\end{aligned}
\end{equation}
Then,
\begin{equation}
	\begin{aligned}
	\partial_s Y &= U, \qquad &&Y(t,t,x,w) = x, \\
	\partial_s U &= \lambda (g + u(Y,s) - U), \qquad  &&U(0,t,x,w) = w, \qquad	U(t,t,x,w) = W(t,x,w).
	\end{aligned}
\end{equation}
Note that by \eqref{eq:rhoByW} 
\[
	\int_{\IR^3} f(t,x,v) \dd v = e^{3 \lambda t} \int_{\IR^3} f_0 (Y,w) \det \nabla_w W \dd w.
\]
We compute
\[
	\partial_{x_i} \det \nabla_w W = \tr ( \op{adj} \nabla_w W \nabla_w \partial_{x_i} W)
	 = \det \nabla_w W \tr ( (\nabla_w W)^{-1} \nabla_w \partial_{x_i} W).
\]
Thus, for any multiindex $\beta$,
\begin{equation}	
	\label{eq:derivativeThreeTermes}
\begin{aligned}
	 \partial_{x_i} \int_{\IR^3} v^\beta f \dd v 
	 &= e^{3 \lambda t} \int_{\IR^3} \partial_{x_i} (W^\beta)  f_0 (Y,w)  \det \nabla_w W \dd w \\
	 {} & + e^{3 \lambda t} \int_{\IR^3} W^\beta \nabla_x f_0 (Y,w) \cdot \partial_{x_i} Y  \det \nabla_w W \dd w \\
	 {} & + e^{3 \lambda t} \int_{\IR^3} W^\beta  f_0 (Y,w) \det \nabla_w W \tr ( (\nabla_w W)^{-1} \nabla_w \partial_{x_i} W) \dd w \\
	 & =: A_1 + A_2 + A_3.
\end{aligned}
\end{equation}

We notice that
\[
	W(t,x,w) = V(t,0,X(0,t,x,W(t,x,w)),V(0,t,x,W(t,x,w))) = V(t,0,Y(0,t,x,w),w).
\]
Hence, for all $(Y(0,t,x,w),w) \in \supp f_0$, estimate \eqref{eq:boundOnV}  implies
\begin{equation}
	\label{eq:boundOnW}
	|W(t,x,w)| \leq C.
\end{equation}

Integrating the equation for $U$ yields (analogously to \eqref{eq:integrateCharacteristics})
\[
	Y(s,t,x,w) = x - \int_s^t g + u(\tau,Y) \dd \tau + \lambda^{-1} ( U(t,t,x,w) - U(s,t,x,w)).  
\]
Therefore,
\begin{equation}
	\label{eq:nabla_xYGronwall}
	\nabla_x Y(s,t,x,w) = \op{Id} - \int_s^t \nabla u(\tau,Y) \nabla_x Y \dd \tau + \lambda^{-1} ( \nabla_xU(t,t,x,w) - \nabla_x U(s,t,x,w)).
\end{equation}

We claim that
\begin{equation}
	\label{eq:velocitiesControlledByPositions}
	|\nabla_x U(s,t,x,w)| \leq 2 \|\nabla u\|_{L^\infty((0,T_0) \times \IR^3)} | \nabla_x Y(s,t,x,w)|.
\end{equation}
Indeed, with 
\begin{align}
	a(s) &:= | \nabla_x Y(s,t,x,w)|,  \\
	b(s) &:=|\nabla_x \partial_s Y(s,t,x,w)|, \\
	\alpha(s) &:= \|\nabla u(s,\cdot)\|_{L^\infty(\IR^3)},
\end{align}
this follows from Lemma \ref{lem:ODEEstimates}\ref{it:ODEEstimates2} using \eqref{eq:lambdaLarge}.

We use estimate \eqref{eq:velocitiesControlledByPositions} in equation \eqref{eq:nabla_xYGronwall} to get
\[
	 a(s)  \leq 1 + \int_s^t \alpha(\tau) a(\tau) \dd \tau 
	+ \frac{2 \| \alpha \|_{L^\infty(0,T)}}{\lambda} (a(t) + a(s)).
\]
Since $a(t) = 0$ and equation \eqref{eq:lambdaLarge} implies $4 \| \alpha \|_{L^\infty(0,T)} \leq \lambda$, we have
\[
	a(s) \leq 2 + 2 \int_s^t \alpha(\tau) a(\tau) \dd \tau.
\]
Therefore, \eqref{eq:uInftyT_ast} implies for all $0 \leq s \leq t \leq T$
\begin{equation}
	\label{eq:nablaxY}
	| \nabla_x Y(s,t,x,w) | = a(s) \leq C.
\end{equation}
Moreover, by \eqref{eq:velocitiesControlledByPositions}, \eqref{eq:uInftyT_ast}, and \eqref{eq:nablaxY}
\begin{equation}
	\label{eq:nablaxW}
	| \nabla_x W(t,x,w) | = |\nabla_x U(t,t,x,w)| \leq C.
\end{equation}

We want to estimate $\nabla_x \det \nabla_w W$.
We compute
\begin{equation}
	\label{eq:derivDet}
	\partial_{x_i} \det \nabla_w W = \tr ( \op{adj} \nabla_w W \nabla_w \partial_{x_i} W)
	 = \det \nabla_w W \tr ( (\nabla_w W)^{-1} \nabla_w \partial_{x_i} W).
\end{equation}
By \eqref{eq:LipschitzV}, we have
\begin{equation}
	|(\nabla_w W(t,x,w)^{-1}| = |(\nabla_v V(0,t,x,W(0,t,x,w))| \leq C e^{\lambda t}.
\end{equation}
	Thus, using also \eqref{eq:JacobianW}, we find
\begin{equation}
	\label{eq:derivDetEst}
	|\partial_{x_i} \det \nabla_w W| \leq \det \nabla_w W  | (\nabla_w W)^{-1}| | \nabla_w \partial_{x_i} W|
	\leq C e^{-3 \lambda t} e^{\lambda t} | \nabla_w \partial_{x_i} W|.
\end{equation}
In order to estimate $| \nabla_w \partial_{x_i} W|$ we further analyze the characteristics $Y$  and $U$ defined in \eqref{eq:defY}.
Fix $t$, $x$, and $w$ and denote 
\begin{align}
	a(s) &:= |\nabla_w Y(s,t,x,w)|, \\
	b(s) &:= |\nabla_w U(s,t,x,w)|, \\
	\alpha(s) &:= \|\nabla u(s,\cdot) \|_{L^\infty(\IR^3)}.
\end{align}
Then, the assumptions of Lemma \ref{lem:ODEEstimates}\ref{it:ODEEstimates1} are satisfied with $\beta = 0$. 

Thus,
\[
	b(t) \leq  \exp \left(\int_0^t -\lambda + 2 \|\nabla u(s,\cdot) \|_{L^\infty(\IR^3)} \dd s \right),
\]
and
\begin{equation}
	\label{eq:nablawY}
	|\nabla_w Y(s,t,x,w)| =  a(s) \leq \frac{2}{\lambda} b(s) \leq \frac{C}{\lambda} e^{-\lambda s}.
\end{equation}

Next, we consider the second derivative.
We denote 
\begin{align}
	a(s) &:= |\nabla_w \nabla_x Y(s,t,x,w)|, \\
	b(s) &:= |\nabla_w \nabla_x U(s,t,x,w)|, \\
	\alpha(s) &:= \|\nabla u(s,\cdot) \|_{L^\infty(\IR^3)}, \\
	\beta & := 4 M(T)^3 \|\nabla^2 u \|_{L^\infty([0,T] \times \IR^3)},
\end{align}
with $M$ as in \eqref{eq:estimateRho}.
Then, using the estimates for  $|\nabla_x Y|$ and $| \nabla_w Y|$ from \eqref{eq:nablaxY} and \eqref{eq:nablawY} respectively,
\begin{align}
	\dot{a} &\geq -b, \\
	\dot{b} & \leq \lambda (\|\nabla^2 u\|_\infty |\nabla_x Y| | \nabla_w Y| + \|\nabla u\|_\infty a - b) 
	\leq  \lambda ( \alpha  a - b) + \beta e^{-\lambda s} .
\end{align}
Hence, the assumptions of Lemma \ref{lem:ODEEstimates}\ref{it:ODEEstimates1} are satisfied.
Since $b(0) = 0$, Lemma \ref{lem:ODEEstimates}\ref{it:ODEEstimates1} yields
\begin{equation}
	\label{eq:nabla2W}
	|\nabla_w \nabla_x W(0,t,x,w) | = b(t) \leq C \frac{2 \beta}{\lambda} e^{- \lambda t} 
	\leq  \frac{C}{\lambda} e^{-\lambda t} \|\nabla^2 u \|_{L^\infty((0,T) \times \IR^3)} .
\end{equation}

Inserting this in \eqref{eq:derivDetEst}, we find
\begin{equation}
	\label{eq:derivDetFinal}
	|\partial_{x_i} \det \nabla_w W|  \leq \frac{C}{\lambda}  e^{-3 \lambda t} \|\nabla^2 u \|_{L^\infty((0,T) \times \IR^3)}.
\end{equation}
We recall the definition of $A_1$, $A_2$, and $A_3$ from equation \eqref{eq:derivativeThreeTermes}. 
Using \eqref{eq:JacobianW} and \eqref{eq:nablaxW} yields
\[
	A_1 \leq C(\beta).
\]
Estimates \eqref{eq:JacobianW}, \eqref{eq:boundOnW}, and \eqref{eq:nablaxY} imply
\[
	A_2 \leq C(\beta).
\]
Finally,  \eqref{eq:boundOnW} and \eqref{eq:derivDetFinal} yield
\[
	A_3 \leq \frac{C(\beta)}{\lambda} e^{-3 \lambda t} \|\nabla^2 u \|_{L^\infty((0,T) \times \IR^3)}.
\]
Inserting these bounds on $A_i$ in \eqref{eq:derivativeThreeTermes} we have.
\begin{equation}
	\label{eq:nablaRhoAlmost}
	\Big\| \nabla_x \int_{\IR^3} v^\beta f(t,x,v) \dd v \Big\|_{L^\infty(\IR^3)} 
	\leq C(\beta) \left( 1+ \frac{1}{\lambda} \|\nabla^2 u \|_{L^\infty((0,T) \times \IR^3)} \right).
\end{equation}
Since the support of $f$ in $x$ is controlled by \eqref{eq:boundOnX}, we also have for any $1 \leq p \leq \infty$
\begin{equation}
	\label{eq:nablaRhoL^p}
	\Big\| \nabla_x \int_{\IR^3} v^\beta f(t,x,v) \dd v \Big\|_{L^p(\IR^3)} 
	\leq C(\beta) (1+T)\left( 1+ \frac{1}{\lambda} \|\nabla^2 u \|_{L^\infty((0,T) \times \IR^3)} \right).
\end{equation}

In order to control $\|\nabla^2 u \|_{L^\infty((0,T)\times\IR^3)}$, the Brinkman equations in \eqref{eq:Vlasov Stokes} and regularity theory for the Stokes equations yield
\begin{align}
	\| \nabla^3 u \|_{L^p(\IR^3)} & \leq \|\nabla (\rho(u-\bar{V}))\|_{L^p(\IR^3)} \\
	& \leq \|\rho \|_{L^p(\IR^3)} \|\nabla u \|_{L^\infty(\IR^3)} + \|\nabla \rho \|_{L^p(\IR^3)} \|u \|_{L^\infty(\IR^3)} 
	+ \|\nabla (\rho \bar{V}))\|_{L^p(\IR^3)}.
\end{align}
Note that both $\nabla \rho$ and $\nabla (\rho \bar{V})$ are of the form of the left hand side in \eqref{eq:nablaRhoL^p}.
Therefore, using also Sobolev embedding together with \eqref{eq:uInftyT_ast} and \eqref{eq:rhoInfty} yields
\[
	\| \nabla^2 u \|_{L^\infty(0,T;C^{2,\alpha})}  \leq C \left( 1+ \frac{1}{\lambda} \|\nabla^2 u \|_{L^\infty((0,T) \times \IR^3)} \right).
\]
This implies \eqref{eq:nabla^2U} for $\lambda$ sufficiently large.

Inserting \eqref{eq:nabla^2U} in \eqref{eq:nablaRhoAlmost} proves \eqref{eq:nablaRho}.
The missing estimate for the time-derivative in \eqref{eq:rhoLipschitz} follows from the Vlasov-Stokes equations \eqref{eq:Vlasov Stokes} and \eqref{eq:nablaRho}.
\end{proof}

\begin{remark}
	One might wonder, whether the complicated splitting in \eqref{eq:derivativeThreeTermes} is really needed.
	Indeed, we also have 
	\begin{align}
		\partial_{x_i} \int_{\IR^3} v^\beta f \dd v &= \partial_{x_i} e^{3 \lambda t} \int_{\IR^3} v^\beta f_0(X,V) \dd v \\
		&= e^{3 \lambda t} \int_{\IR^3} v^\beta \nabla_x f_0(X,V) \partial_{x_i} X + \nabla_v f_0(X,V) \partial_{x_i} V \dd v,
	\end{align}
	an expression that involves only two terms and in particular does not involve any second derivatives.
	However, it turns out, that both $\partial_{x_i} X$ and $\partial_{x_i} V$ blow up as $\lambda \to \infty$.
	Therefore, estimating both terms individually in the above expression cannot lead to the assertion.
\end{remark}

%% file: ParticleVelocityFluidVelocity.tex
\section{Proof of the Convergence result}

\subsection{Error estimates for the particle and fluid velocities}

Recall the definition of $\tilde{u}_\lambda$ from \eqref{eq:uTilde}, which can be viewed as intermediate between $u_\lambda$ and $u_\ast$ 
defined by \eqref{eq:Vlasov Stokes} and \eqref{eq:limitEquation} respectively. As a first step to show smallness of $u_\lambda - u_\ast$ 
(and also $\rho_\lambda - \rho_\ast$), we will show smallness of $u_\lambda - \tilde{u}_\lambda$. 
Comparing the PDEs that $u_\lambda$ and $\tilde{u}_\lambda$ fulfill, we observe that we have to prove smallness of $\rho(\bar{V} -  u_\lambda - g)$.
This is almost what we do in the proof of the lemma below. Indeed, it turns out that it is more convenient to consider the
error term $\Phi = \bar{V} -  \tilde{u}_\lambda - g$ instead of $\bar{V} -  u_\lambda - g$ because we control the time derivative of $\tilde{u}$.
Then, we are able to prove smallness of $u_\lambda - \tilde{u}_\lambda$
using energy identities for $\Phi$ and $u_\lambda - \tilde{u}_\lambda$ analogous to \eqref{eq:energyFluid} and \eqref{eq:energyParticles}.

\begin{lemma}
	\label{lem:uTilde}
	Assume $T < T_\ast$ and let $\tilde{u}_\lambda$ be defined as in \eqref{eq:uTilde}. Then, there exist $\lambda_0$ such that for all $\lambda \geq \lambda_0$
	\begin{align}
	\label{eq:estTildeU}
	\|\tilde u \|_{W^{1,\infty}((0,T_0)\times\IR^3)} &\leq C, \\
	\label{eq:diffUTildeU}
	\|\tilde u(t,\cdot) - u(t,\cdot)\|^2_{W^{1,\infty}(\IR^3)} &\leq C \left( e^{-c \lambda t} + \frac{1}{\lambda} \right).
\end{align}
\end{lemma}

\begin{proof}
Again, we consider only $\lambda > \lambda_0$ whith $\lambda_0$ as in Lemma \ref{lem:APriori}.
Then, Lemma \ref{lem:APriori} implies that we control the $L^\infty$-norms of $\rho$ and $\partial_t \rho$ and the support of $\rho$.
Thus, \eqref{eq:estTildeU} follows from regularity theory for the Stokes equations.

We define
\begin{align}
	\Phi &:= \bar{V} - \tilde{u} -  g, \\
	 Z &:= u - \tilde{u}.
\end{align}
Then, 
\begin{align}
	-\Delta Z + \nabla p + (Z-\Phi)\rho =  0, \qquad \dv Z =0.
\end{align}
Therefore
\begin{equation}
	\label{eq:energyZ}
	\| \nabla Z \|^2_{L^2(\IR^3)}  =  (Z,\Phi-Z)_{L^2_\rho(\IR^3)}.
\end{equation}
We compute
\begin{align}
	\partial_t (\rho \bar{V})  &= -  \int_{\IR^3} v \cdot \nabla_x f v \dd v + \lambda \rho(g + u -\bar{V}) 
	= - \int_{\IR^3} v \cdot \nabla_x f v \dd v + \lambda \rho(Z - \Phi),  \\
	\label{eq:del_tRhoPhi} \partial_t (\rho \Phi) &= \partial_t (\rho \bar{V})- \partial_t (\rho \tilde{u}) 
			=  \lambda \rho(Z - \Phi) - \int_{\IR^3} v \cdot \nabla_x f v \dd v -  \partial_t (\rho \tilde{u}).
\end{align}
Note that \eqref{eq:estTildeU} and the bound on $\bar{V}$ from Lemma \ref{lem:APriori} imply
 that $\Phi(t,\cdot)$ is uniformly bounded in $L^\infty(\IR^3)$ up to time $T$.
Thus, we use \eqref{eq:del_tRhoPhi}, \eqref{eq:energyZ}, and the estimates from Lemma \ref{lem:APriori}, \eqref{eq:estTildeU}, and
 Lemma \ref{lem:frictionCoercive} to obtain
\begin{align*}
	\partial_t \frac{1}{2} \| \Phi\|_{L^2(\rho)}^2 &= \int_{\IR^3} \partial_t (\rho \Phi) \cdot \Phi - \frac{1}{2} \partial_t \rho |\Phi|^2 \dd x \\
	&= \lambda \int_{\IR^3} \rho \Phi \cdot ( Z - \Phi) \dd x  -\int_{\IR^3 \times \IR^3} \hspace{-1em}  v \cdot \nabla_x f v \cdot \Phi \dd v \dd x \\ 
	&  - \int_{\IR^3}  \partial_t (\rho \tilde{u}) \cdot \Phi\dd x - \frac{1}{2} \int_{\IR^3}  \partial_t \rho |\Phi|^2 \dd x  \\
	&\leq - \lambda \|\nabla Z\|_{L^2(\IR^3)}^2 - \lambda \|Z - \Phi\|_{L^2_\rho(\IR^3)}^2 + C \\
	&\leq - c \lambda \|\Phi\|_{L^2_\rho(\IR^3)}^2 + C.
\end{align*}
Therefore, we have
\[
	\| \Phi\|_{L^2_\rho(\IR^3)}^2 \leq  C\left( e^{-c \lambda t} + \frac{1}{\lambda} \right).
\]

By the energy identity for the Brinkman equations \eqref{eq:energyZ}, it follows
\[
		\|\nabla Z\|^2_{L^2 (\IR^3)} + \| Z\|^2_{L^2_\rho(\IR^3)} \leq C\left( e^{-c \lambda t} + \frac{1}{\lambda} \right).
\]
Regularity theory for Stokes equations implies
\begin{align}
	\| \nabla^2 Z \|^2_{L^2(\IR^3)} \leq 2 \|\rho Z\|^2_{L^2(\IR^3)} + 2 \|\rho \Phi \|^2_{L^2(\IR^3)}  
	&\leq 2 \|\rho \|^2_{L^3(\IR^3)} \|Z\|^2_{L^6(\IR^3)} +  2 \|\Phi\|^2_{L^2_\rho(\IR^3)} \|\rho\|_{L^\infty(\IR^3)} \\
	&\leq C\left( e^{-c \lambda t} + \frac{1}{\lambda} \right).
\end{align}
Thus, using Sobolev embedding, 
\begin{equation}
	\label{eq:ZLinfty}
	\| Z \|_{L^\infty(\IR^3)}^2 \leq C\left( e^{-c \lambda t} + \frac{1}{\lambda} \right).
\end{equation}

Taking $\lambda_0 \geq 1$ and using again \eqref{eq:del_tRhoPhi} yields 
\[
	\partial_t (\rho\Phi) \leq - \rho \Phi + C( \lambda e^{-c \lambda t} + \sqrt{\lambda}).
\]
Thus,
\[
	\| \rho \Phi \|_{L^\infty(\IR^3)}^2 \leq C\left( e^{-c \lambda t} + \frac{1}{\lambda} \right),
\]
which again yields smallness of $Z$ in even better norms.
More precisely, for $p \geq 2$
\begin{align}
	\| \nabla^2 Z \|^2_{L^p(\IR^3)} \leq C \|\rho Z\|^2_{L^p(\IR^3)} + C \|\rho \Phi \|^2_{L^p(\IR^3)} 
	&\leq C \|\rho \|^2_{L^p(\IR^3)} \|Z\|^2_{L^\infty(\IR^3)} +   C \|\rho \Phi\|^2_{L^\infty(\IR^3)} \\
	&\leq C \left( e^{-c \lambda t} + \frac{1}{\lambda} \right).
\end{align}
In particular,
\[
	\| Z\|^2_{W^{1,\infty}} \leq C \left( e^{-c \lambda t} + \frac{1}{\lambda} \right).
\]
By definition of $Z$, this proves \eqref{eq:diffUTildeU}.
\end{proof}

%% file: Convergence.tex
\subsection{Convergence for times $\bm{T<T_\ast}$}

We want to prove $\rho_\lambda \to \rho_\ast$ as $\lambda \to \infty$, where $\rho_\ast$ is the solution to \eqref{eq:limitEquation}. 
By the a priori estimate  from Lemma \ref{lem:APriori}, 
we have that $\rho_\lambda$ is uniformly bounded in $W^{1,\infty}((0,T_0)\times \IR^3)$ for times $T_0 < T_\ast$ defined in \eqref{eq:defT_ast}.
Hence, we can extract strongly convergent subsequences in $C^{0,\alpha}((0,T_0)\times \IR^3)$ for all $\alpha < 1$.
It remains to prove that any limit of these subsequences is $\rho_\ast$. To this end we will show  that $\rho_\lambda$ converges to $\rho_\ast$
in a weaker sense by using again the characteristics.

We note that
\begin{equation}
	\label{eq:rho_astByX_ast}
	\rho_\ast (t,x) = \rho_0 (X_\ast(0,t,x)) = \int_{\IR^3} f_0(X_\ast(0,t,x),v) \dd v,
\end{equation}
where $X_\ast(s,t,x)$ is defined as the solution to
\begin{align}
	\partial_s X_\ast(s,t,x) &= g + u_\ast(s,X_\ast(s,t,x)), \\
	X_\ast(t,t,x) &= x.
\end{align}

We have seen in \eqref{eq:almostTransportedByU} that for large values of  $\lambda$, the particles are almost transported by $u_\lambda +g$.
Moreover, in Lemma \ref{lem:uTilde}, we have seen that the fluid velocity $u_\lambda$
is close to $\tilde u_\lambda$, which roughly speaking is the fluid velocity corresponding to the limit equation \eqref{eq:limitEquation}.

In order to compare $\rho_\lambda $ to $\rho_\ast$, we want to use the formula for $\rho_\lambda$ from Lemma \eqref{lem:biLipschitzV},
\begin{equation}
	\label{eq:rhoByW2}
	\rho_\lambda(t,x) = \int_{\IR^3} e^{ 3 \lambda t} f_0(X_\lambda(0,t,x,W_\lambda(t,x,w)),w) \det \nabla_w W_\lambda(t,x,w)  \dd w.
\end{equation}
Provided $X_\lambda(0,t,x,W_\lambda(t,x,w))$ is close to $X_\ast(0,t,x)$ independently of $w$, the right hand sides of
\eqref{eq:rho_astByX_ast} and \eqref{eq:rhoByW2} look very similar.
However, we lack information on the Jacobian $\det \nabla_w W_\lambda(t,x,w)$.
 We know that $e^{ 3 \lambda t}  \det \nabla_w W_\lambda(t,x,w)$
is uniformly bounded (for small times $t$ and large values of $\lambda$, cf. Lemma \ref{lem:biLipschitzV} and Lemma \ref{lem:APriori}),
but we do not know whether it tends to $1$ in the limit $\lambda \to \infty$.

To avoid dealing with this Jacobian, we also integrate over a small set in position space.
To this end, let $\Psi_\lambda(t,\xi) := (X_\lambda(t,0,\xi),V_\lambda(t,0,\xi))$ with $\xi = (x,v)$. 
Then, using the characteristic equations \eqref{eq:characteristics},
\[
	\partial_t \nabla \Psi_\lambda = \nabla \Psi_\lambda
\begin{pmatrix}
  0 & \mathrm{Id} \\
  \lambda \nabla u & -\lambda \mathrm{Id} 
 \end{pmatrix}
 .
\]
Hence,
\[
	\partial_t \det \nabla \Psi_\lambda =   \det \nabla \Psi_\lambda \tr \left( (\nabla \Psi_\lambda)^{-1} \nabla \Psi_\lambda
\begin{pmatrix}
  0 & \mathrm{Id} \\
  \lambda \nabla u & -\lambda \mathrm{Id} 
 \end{pmatrix} \right) = - 3 \lambda \det \nabla \Psi_\lambda.
\]
Thus,
\[
	\det \nabla \Psi_\lambda(t,\xi) = e^{- 3 \lambda t}.
\]
Therefore, for $\Omega \subset \IR^3$ measurable,
\begin{equation}
	\label{eq:intRho_lambda}
	\begin{aligned}
		\int_{\Omega} \rho_\lambda(t,y) \dd y = \int_{\Omega} \int_{\IR^3} e^{3 \lambda t} f_0(\Psi_\lambda^{-1}(y,v)) \dd v \dd y  = \int_{\Psi_\lambda^{-1}(\Omega \times \IR^3)} f_0(y,v) \dd v \dd y.
 	\end{aligned}
 \end{equation}
 
  On the other hand, since $u_\ast$ is divergence-free, we observe that
\begin{equation}
	\label{eq:intRho_ast}
	\int_{\Omega} \rho_\ast(t,y) \dd y = \int_{\Omega} \rho_0(X_\ast(0,t,y)) \dd y =  
	\int_{X(0,t,\Omega)} \rho_0(y) \dd y = \int_{X(0,t,\Omega) \times \IR^3} \hspace{-1em} f_0(y,v) \dd y \dd v .
\end{equation}
Now, we have to compare the right hand sides of \eqref{eq:intRho_ast} and \eqref{eq:intRho_lambda}.

It is convenient to choose $\Omega$ to be a cube.
We denote by $\mathcal{Q}_\delta$ the the set of all cubes $Q \subset \IR^3$ of length $\delta$. We define
\[
	d_{\lambda,\delta}(t) := \sup_{Q \in \mathcal{Q}_\delta} \left| \fint_Q \rho_\lambda(t,y) - \rho_\ast(t,y) \dd y \right|.
\]
We will show that 
\begin{equation}
	\label{eq:weakL^1Cubes}
	 \lim_{\lambda \to \infty} \lim_{\delta \to 0}  d_{\lambda,\delta}(t) = 0 \qquad \text{for  all} ~ t < T_\ast.
\end{equation}
This implies $\rho_\lambda(t,\cdot) \to \rho_\ast(t,\cdot)$ weakly-* in $L^\infty(\IR^3)$ because we already have uniform boundedness
by Lemma \ref{lem:APriori}.

For the proof of \eqref{eq:weakL^1Cubes} in Proposition \ref{pro:ConvergenceCube}, we essentially need three ingredients.
First, we will show in Lemma \ref{lem:d_lambda,delta} that $d_{\lambda,\delta}$ is controlled by
$|X_\lambda - X_\ast|$.
Second, we will show in Lemma \ref{lem:uTildeUAst} that $\tilde{u}_\lambda- u_\ast$ is controlled by $d_{\lambda,\delta}$. Finally, we use that the particle trajectories $X_\lambda$ are almost the ones, which one get 
from a transport velocity $\tilde{u}_\lambda + g$. This last ingredient is due to \eqref{eq:almostTransportedByU} and 
Lemma \ref{lem:uTilde}.

\begin{lemma}
	\label{lem:d_lambda,delta}
	Let $T_0 < T_\ast$. Then, there exist constants $C_1$  and $\lambda_0$ such that for all $\lambda > \lambda_0$ and all $t < T_0$
	\[
		d_{\lambda,\delta}(t) 
		\leq C \left(\sup_{(x,v) \in \supp f_0} |X_\lambda(t,0,x,v) - X_\ast(t,0,x)| + \delta + \frac{1}{\delta \lambda} \right).
	\]
\end{lemma}

\begin{proof}
	Let $Q \in \mathcal{Q}_\delta$.
	Let $\Psi_\lambda(t,y,v) := (X_\lambda(t,0,y,v),V_\lambda(t,0,y,v))$. Recall from \eqref{eq:intRho_ast} and \eqref{eq:intRho_lambda}
\begin{align}
	\label{eq:rho_astByf_0}
		\int_{Q} \rho_\ast(t,y) \dd y & =  \int_{X(0,t,Q)} \rho_0(y) \dd y, \\
		\label{eq:rho_lambdaByf_0}
		\int_{Q} \rho_\lambda(t,y) \dd y &= \int_{\Psi_\lambda^{-1}(Q \times \IR^3)} f_0(y,v) \dd y \dd v.
\end{align}
We want to replace the right hand side of \eqref{eq:rho_lambdaByf_0} by an integral of $\rho_0$ 
to compare its value to the right hand side of \eqref{eq:rho_astByf_0}. To this end,
we have to replace the set $\Psi_\lambda^{-1}(Q \times \IR^3)$ by a set of the form $\Omega \times \IR^3$.
We define
\begin{equation}
	\label{eq:defOmega}
\begin{aligned}
	\Omega &:= \{ X(0,t,z,w) \colon (z,w) \in \Psi(\supp f_0) \cap (Q \times \IR^3 ) \}  \\
						&= \{ y \in \IR^3 \colon \text{there is} ~ v \in \IR^3 ~ \text{with} ~ (y,v) \in \supp f_0, X_\lambda(t,0,y,v) \in Q \}.
\end{aligned}
\end{equation}
 	Then, we claim
\begin{equation}
	\label{eq:Inclusions}
	\Psi_\lambda^{-1}(Q \times \IR^3) \cap \supp f_0 \subset \left( \Omega \times \IR^3 \right) \cap \supp f_0
	\subset \Psi_\lambda^{-1}(Q_{C\lambda^{-1}} \times \IR^3),
\end{equation}
where $C$ is a constant independent of $\delta$ (and $\lambda$), and
\[
	Q_{C\lambda^{-1}} := \bigcup_{y \in Q} B_{C\lambda^{-1}}(y).
\]
The first inclusion in \eqref{eq:Inclusions} follows from the definition of $\Omega$. 
To prove the second inclusion,
let $(y,v) \in \supp f_0 \cap ( \Omega  \times \IR^3 )$. Then, by definition of
$\Omega$, there exists $\tilde{v} \in \IR^3$ such that $(y,\tilde{v}) \in \supp f_0$ and 
$X_\lambda(t,0,y,\tilde{v}) \in Q$.
From \eqref{eq:almostTransportedByU} and the fact that the support of $f_\lambda$ is uniformly bounded up to time $T_0$ by Lemma \ref{lem:APriori}, we know that
\begin{align}
	|X_\lambda(t,0,y,v) - X_\lambda(t,0,y,\tilde{v})| 
	&\leq \frac{C}{\lambda} + \int_0^t |u(s,X_\lambda(s,0,y,v)) - u(s,X_\lambda(s,0,y,\tilde{v}))| \dd s \\
								&\leq \frac{C}{\lambda} + \int_0^t \|\nabla u\|_{L^\infty} |X_\lambda(s,0,y,v) - X_\lambda(s,0,y,\tilde{v})|\dd s.
\end{align}
Using the estimate for $\nabla u$ from Lemma \ref{lem:APriori} yields
\begin{equation}
	\label{eq:smallDependenceOnInitialVelocity}
		|X_\lambda(t,0,y,v) - X_\lambda(t,0,y,\tilde{v})| \leq \frac{C}{\lambda} e^{C t}.
\end{equation}
Therefore,  $X_\lambda(t,0,y,v) \in Q_{C\lambda^{-1}}$ and thus
$(y,v) \in \Psi_\lambda^{-1}(Q_{C\lambda^{-1}} \times \IR^3)$.
From \eqref{eq:Inclusions} it follows
\begin{equation}
	\label{eq:cylinderRemaining}
 	\begin{aligned}
 	 	&\left| \int_{\Psi_\lambda^{-1}(Q \times \IR^3)} f_0(y,v) \dd y \dd v - 
 	 	\int_{\Omega \times \IR^3} f_0(y,v) \dd y \dd v \right| \\  
 	 	&\leq \int_{\Psi_\lambda^{-1}((Q_{C\lambda^{-1}} \backslash Q )\times \IR^3)} f_0(y,v) \dd y \dd v \\
 	 	& = \int_{Q_{C\lambda^{-1}} \backslash Q } \rho_\lambda (t,y) \dd y \\
 	 	& \leq \|\rho_\lambda(t,\cdot)\|_{L^\infty(\IR^3)} |Q_{C\lambda^{-1}} \backslash Q| \\
 	 	&\leq C \frac{\delta^2}{\lambda}.
 	\end{aligned}
\end{equation}
Combining \eqref{eq:rho_astByf_0}, \eqref{eq:rho_lambdaByf_0}, and \eqref{eq:cylinderRemaining} yields
\begin{equation}
	\label{eq:integralRho_0}
	\left| \int_{Q} \rho_\lambda(t,y) - \rho_\ast(t,y) \dd y \right| \leq 
	\left| \int_{X_\ast(0,t,Q)} \rho_0(y) \dd y - \int_{\Omega} \rho_0(y) \dd y \right| + C \frac{\delta^2}{\lambda}.
\end{equation}
To estimate the right hand side, we note that 
\begin{equation}
	\label{eq:X_astVolumePreserving}
	|X_\ast(0,t,Q)| = |Q| = \delta^3,
\end{equation}
 since $\dv u_\ast = 0$.	
We want to show that $|\Omega| \approx |Q|$.
To this end, we define $\tilde{X}_\lambda$ to be the solution to
\begin{align}
	\partial_s {\tilde{X}}_\lambda(s,t,x) &= u_\lambda(s,\tilde{X}_\lambda(s,t,x)), \\
	\tilde{X}_\lambda(t,t,x) &= x.
\end{align}
Then, using \eqref{eq:almostTransportedByU}, we have for all $(x,v) \in \supp f_0$
\begin{align}
	|\tilde{X}_\lambda(t,0,x) - X_\lambda(t,0,x,v)| 
	&\leq \frac{C}{\lambda} + \int_0^t |u_\lambda(s,\tilde{X}_\lambda(s,t,x)) - u_\lambda(s,X_\lambda(s,t,x,v))| \dd s \\
	&\leq \frac{C}{\lambda} + \int_0^t \|\nabla u_\lambda \|_{L^\infty(\IR^3)} | \tilde{X}_\lambda(s,t,x) - X_\lambda(s,t,x,v)| \dd s.
\end{align}
Gronwall implies
\[
	|\tilde{X}_\lambda(t,0,x) - X_\lambda(t,0,x,v)|  \leq \frac{C}{\lambda}.
\]
Thus, 
\[
	\tilde{X}^{-1}_\lambda(t,0,I_{C \lambda^{-1}}(Q)) \subset \Omega \subset \tilde{X}^{-1}_\lambda(t,0,Q_{C\lambda^{-1}}),
\]
where
\[
	I_{C \lambda^{-1}}(Q) := \{ y \in Q \colon B_{C \lambda^{-1}}(y) \subset Q \}.
\]
 Since $\dv u_\lambda = 0$, we have that $\tilde{X}_\lambda$ is volume preserving as well. Therefore, using also \eqref{eq:X_astVolumePreserving}
\begin{equation}
	\label{eq:VolumeDifference}
	||\Omega| - |X_\ast(0,t,Q)|| \leq | Q_{C\lambda^{-1}} \backslash I_{C \lambda^{-1}}(Q)|\leq  C \frac{\delta^2}{\lambda}.
\end{equation}

We observe that for any function $g \in W^{1,\infty}(\IR^3)$ and measurable sets $E,F \subset \IR^3$
\begin{equation}
	\label{eq:integralSameFunctionDifferentSets}
	\left|\int_E g  \dd x - \int_F g \dd x\right| \leq ||E| - |F|| \|g\|_{L^\infty} + \min\{|E|,|F|\} \|\nabla g\|_{L^\infty} \sup \{|x-y| \colon x \in E, y \in F\}.
\end{equation}
Indeed, using the first term on the right hand side, we may assume without loss of generality that $E$ and $F$ are of equal measure. 
Approximating $E$ and $F$ by equisized
cubes further reduces the situation to the estimate for two of these cubes. For these cubes, the statement obviously holds. 

Applying \eqref{eq:integralSameFunctionDifferentSets} together with \eqref{eq:X_astVolumePreserving} and \eqref{eq:VolumeDifference} yields
\begin{equation}
	\label{eq:estimateRho_0Integral}
\begin{aligned}
	&\left| \int_{X_\ast(0,t,Q)} \rho_0(y) \dd y - \int_{\Omega} \rho_0(y) \dd y \right| \\
	 &\leq ||\Omega| - |X_\ast(0,t,Q)||\|\rho_0\|_{L^\infty} 
	+ \delta^3 \sup \{|y-z| \colon y \in \Omega, z \in X_\ast(0,t,Q) \} \|\nabla \rho_0\|_{L^\infty} \\
	& \leq  C \frac{\delta^2}{\lambda} +  C \delta^3  \bigg( \sup_{y \in \Omega} \dist (y,X_\ast(0,t,Q)) 
	+ \op{diam} (X_\ast(0,t,Q)) \bigg).	 
\end{aligned}
\end{equation}
We need to estimate the second term on the right hand side.
To this end, recall the definition of the set $\Omega$ from \eqref{eq:defOmega}. 
For any $y \in \Omega$, we find $(x,v) \in \supp f_0$ such that $p:=X_\lambda(t,0,y,v) \in Q$.
Define $z = X_\ast(0,t,p) \in X_\ast(0,t,Q)$. Then
\begin{equation}
	\label{eq:estimateDist}
\begin{aligned}
	|z - y| &= |X_\ast(0,t,p) - X_\ast(0,t,X_\ast(t,0,y)| \\
	&\leq \|\nabla X_\ast (0,t,\cdot)\|_{L^\infty(\IR^3)} |X_\lambda(t,0,y,v) - X_\ast(t,0,y)| \\
	&\leq \|\nabla X_\ast (0,t,\cdot)\|_{L^\infty(\IR^3)} \sup_{(x,v) \in \supp f_0} |X_\lambda(t,0,x,v) - X_\ast(t,0,x)|.
\end{aligned}
\end{equation}
Observe that
\begin{equation}
	\label{eq:nablaX_ast}
	\|\nabla X_\ast (0,t,\cdot)\|_{L^\infty(\IR^3)} \leq e^{\int_0^t \|\nabla u_\ast(s,\cdot)\| \dd s} \leq C.
\end{equation}
Thus, \eqref{eq:estimateDist} and \eqref{eq:nablaX_ast} imply
\begin{equation}
	\label{eq:estDist}
\begin{aligned}
	\sup_{y \in \Omega} \dist (y,X_\ast(0,t,Q)) 
	&\leq C \sup_{(x,v) \in \supp f_0} |X_\lambda(t,0,x,v) - X_\ast(t,0,x)|.
\end{aligned}
\end{equation}
Note that \eqref{eq:nablaX_ast} also yields
\begin{equation}
	\label{eq:diam}
	\op{diam} (X_\ast(0,t,Q)) \leq \delta \|\nabla X_\ast (0,t,\cdot)\|_{L^\infty(\IR^3)} \leq C \delta.
\end{equation}

Finally, estimates \eqref{eq:estimateRho_0Integral}, \eqref{eq:diam}, and \eqref{eq:estDist} yield
\[
	\left| \int_{X_\ast(0,t,)} \rho_0(y) \dd y - \int_{\Omega} \rho_0(y) \dd y \right| 
	\leq C \frac{\delta^2}{\lambda} + C \delta^3\left(\sup_{(x,v) \in \supp f_0} |X_\lambda(t,0,x,v) - X_\ast(t,0,x)| + \delta\right).
\]
Combining this estimate with \eqref{eq:integralRho_0} finishes the proof.
\end{proof}

\begin{lemma}
	\label{lem:uTildeUAst}
	Let $T_0 < T_\ast$. 
	For $u_\ast$ and $\tilde{u}_\lambda$ as in \eqref{eq:limitEquation} and \eqref{eq:uTilde}, we have for all $\delta \leq 1$ and for all
	$t < T_0$
	\[
		\|\tilde{u}_\lambda(t,\cdot) - u_\ast(t,\cdot)\|_{L^\infty(\IR^3)} \leq C (d_{\lambda,\delta}(t) + \delta).
	\]
\end{lemma}

\begin{proof}
We choose  disjoint cubes $(Q_i)_{j \in \IN} \subset \mathcal{Q}_\delta$ that cover  $\IR^3$ up to a nullset. Define $I \subset \IN$ to be the index set for those cubes that
intersect with the support of either $\rho_\lambda(t,\cdot)$ or $\rho_\ast(t,\cdot)$ and let $(z_i)_{i \in I}$ be the centers of those cubes.
Let $x \in \IR^3$. Then,
\begin{align}
	|\tilde{u}_\lambda(t,x) - u_\ast(t,x)| &= \left| \int_{\IR^3} \Phi(x-y) (\rho_\lambda(t,y) - \rho_\ast(t,y) )\dd y \right| \\
		&\leq \sum_{j \in I} \left| \int_{Q_j} \Phi(x-y) (\rho_\lambda(t,y) - \rho_\ast(t,y) )\dd y \right|,
\end{align}
where $\Phi$ is the fundamental solution of the Stokes equations,
\[
	\Phi(y) = \frac{1}{8 \pi} \left(\frac{\mathrm{Id}}{|y|} + \frac{y \otimes y}{|y|^3} \right). 
\]
Let $I_1 \subset I$ be the index set of those cubes $Q_j$ which contain $x$ or are adjacent to that cube.
Then, $|I_1| \leq 27$ and for  $j \in I_1$ we estimate
\begin{equation}
\begin{aligned}
	 \left| \int_{Q_j} \Phi(x-y) (\rho_\lambda(t,y) - \rho_\ast(t,y) )\dd y \right| 
	 &\leq (\|\rho_\lambda(t,\cdot)\|_{L^\infty(\IR^3)} + \|\rho_\ast(t,\cdot)\|_{L^\infty(\IR^3)} \left| \int_{Q_j} \Phi(x-y) \dd y \right|\\
	 &\leq C \delta^2.
\end{aligned}
\end{equation}

Let $I_2 = I \backslash I_1$. For $h \in L^1(\IR^n)$ and $\Omega \subset \IR^n$ measurable, we use the notation 
\[
	(h)_{\Omega} := \fint_\Omega h \dd x := \frac{1}{|\Omega|} \int_\Omega h \dd x.
\] 
Then, for $j \in I_2$, 
\begin{equation}
\begin{aligned}
	 \left| \int_{Q_j} \Phi(x-y) (\rho_\lambda(t,y) - \rho_\ast(t,y) )\dd y \right| 
	 &\leq |(\Phi(x-\cdot))_{Q_j}| \left| \int_{Q_j} (\rho_\lambda(t,y) - \rho_\ast(t,y) )\dd y \right| \\ 
	 {} &+  \int_{Q_j} |\Phi(x-y) - (\Phi(x-\cdot))_{Q_j}| |\rho_\lambda(t,y) - \rho_\ast(t,y) |\dd y \\
	 & \leq \frac{\delta^3}{|x-z_j|} d_{\lambda,\delta}(t) + \frac{\delta^4}{|x-z_j|^2},
\end{aligned}
\end{equation}
where we used  that we control  $\rho_\lambda(t,\cdot)$ and $\rho_\ast(t,\cdot)$ in $L^\infty(\IR^3)$ by Lemma \ref{lem:APriori}.
Summing over all $j \in I$ yields
\begin{align}
	|\tilde{u}_\lambda(t,x) - u_\ast(t,x)| 
	&\leq C \delta^2 + \sum_{j \in I^2} \frac{\delta^3}{|x-z_j|} d_{\lambda,\delta}(t) + \frac{\delta^4}{|x-z_j|^2} \\
	&\leq C (\delta^2 + \delta + d_{\lambda,\delta}(t)),
\end{align}
where the constant $C$ depends on the spatial support of $\rho_\lambda$ and $\rho_\ast$ which we control uniformly up to time $T_0$ by Lemma \ref{lem:APriori}.
Using $\delta \leq 1$ finishes the proof.
\end{proof}

\begin{proposition}
	\label{pro:ConvergenceCube}
	Let $t < T_\ast$. Then
\[
	\lim_{\delta \to 0} \lim_{\lambda \to \infty} d_{\lambda,\delta}(t) = 0.
\]
\end{proposition}

\begin{proof}
	We define
	\[
			\eta(t) := \sup_{(x,v) \in \supp f_0} |X_\lambda(t,0,x,v) - X_\ast(t,0,x)|.
	\]
	Let $(x,v) \in \supp f_0$. We write again $X_\lambda(t)$ instead of $X_\lambda(t,0,x,v)$ and similar for $X_\ast$.
	We estimate using first \eqref{eq:almostTransportedByU} together with the fact that the support of $f_\lambda$ remains uniformly
	bounded up to time $T_0$, and then applying Lemma \ref{lem:uTilde}, Lemma \ref{lem:uTildeUAst}, and Lemma \ref{lem:d_lambda,delta}
	\begin{align*}
		|X_\lambda(t)) - X_\ast(t)| & \leq \int_0^t |u_\lambda(s,X_\lambda(s)) - u_\ast(s,X_\ast(s))|  \dd s + \frac{C}{\lambda} \\
		& \leq \int_0^t |\tilde{u}_\lambda(s,X_\lambda(s)) - u_\ast(s,X_\ast(s))| + 
		 |\tilde{u}_\lambda(s,X_\lambda(s)) - u_\lambda(s,X_\ast(s))| \dd s + \frac{C}{\lambda} \\
		& \leq \int_0^t \|\tilde{u}_\lambda(s,\cdot) - u_\ast(s,\cdot)\|_{L^\infty(\IR^3)} + 
		\|\nabla u_\ast(s,\cdot)\|_{L^\infty(\IR^3)} |X_\lambda(s) - X_\ast(s)| \dd s + \frac{C}{\lambda} \\
		& \leq C \int_0^t d_{\lambda,\delta}(t) + \delta + |X_\lambda(s) - X_\ast(s)| \dd s  + \frac{C}{\lambda} \\
		& \leq C \int_0^t \eta(t) + \frac{1}{\delta \lambda} + \delta  \dd s + \frac{C}{\lambda}.
	\end{align*}
	Taking the supremum over $(x,v) \in \supp f_0$ yields for $\delta \leq 1$
	\begin{align}
		\eta(t) &\leq C \int_0^t \eta(s) \dd s +C \left( \frac{1}{\delta \lambda} + \delta\right).
	\end{align}
	Gronwall's inequality implies
	\[
		\eta(t) \leq C \left( \frac{1}{\delta \lambda} + \delta\right) e^{C t}.
	\]
	Lemma \ref{lem:d_lambda,delta} yields
	\[
		d_{\lambda,\delta}(t) \leq C \left( \frac{1}{\delta \lambda} + \delta\right) e^{C t}.
	\]
	Taking the limits $\lambda \to \infty$ followed by $\delta \to 0$ finishes the proof.
\end{proof}

Now, we have all the estimates needed to prove the statement of Theorem \ref{th:main} up to times $T < T_\ast$.

\begin{proposition}
	\label{pro:strongConvergence}
	Let $T < T_\ast$. Then, for all $\alpha < 1$,
	\begin{equation}
		\label{eq:ConvergenceRho}
		\rho_\lambda \to \rho_\ast \quad \text{in} ~ C^{0,\alpha}((0,T) \times \IR^3).
	\end{equation}
	Moreover, for all $0 < t < T$,
	\begin{equation}
		\label{eq:CovergenceU}
		u_\lambda \to u_\ast \quad  \text{in} ~ L^\infty((t,T) ; W^{1,\infty}(\IR^3)) ~  \text{and in} ~ L^1((0,T) ; W^{1,\infty}(\IR^3)).
	\end{equation}
\end{proposition}

\begin{proof}
	By Lemma \ref{lem:APriori}, the sequence $\rho_\lambda$ is uniformly bounded in $W^{1,\infty}((0,T) \times \IR^3)$ for large enough $\lambda$.
	Therefore, for any $\alpha < 1$, $\rho_\lambda$  has a subsequence that converges in $C^{0,\alpha}((0,T) \times \IR^3)$ to some 
	function $\sigma$. We need to show $\sigma = \rho_\ast$.
	We claim that for all cubes $Q \subset \IR^3$ and all $t < T$,
	\begin{equation}
		\label{eq:ConvergenceCubes}
		\int_Q \rho_\lambda(t,x) \dd x \to \int_Q \rho_\ast(t,x) \dd x.
	\end{equation}
	Clearly, \eqref{eq:ConvergenceCubes} implies $\sigma = \rho_\ast$. In order to prove \eqref{eq:ConvergenceCubes}, let
	$\eps > 0$. Then, by Proposition \ref{pro:ConvergenceCube}, there exists $\delta_0 > 0$ such that for all $\delta < \delta_0$ 
	and all $x \in \IR^3$
	\begin{equation}
		\label{eq:smallerCubes}
	 \lim_{\lambda \to \infty} d_{\lambda,\delta} = \left| \fint_{Q_{\delta,x}} \rho_\lambda(t,x) - \rho_\ast(t,x) \dd x \right| < \eps.
	\end{equation}
	Up to a nullset, we can write $Q$ as the disjoint union of cubes $Q_i \in \cup_{\delta < \delta_0} \mathcal{Q}_\delta$.
	Thus, since $\eps$ is arbitrary, \eqref{eq:ConvergenceCubes} follows.
	
	In order to prove \eqref{eq:CovergenceU}, we notice that by Lemma \ref{lem:uTilde} it suffices to prove
	\begin{equation}
		\tilde{u}_\lambda \to u_\ast \quad  \text{in} ~ L^\infty((0,T) ; W^{1,\infty}(\IR^3)).
	\end{equation}
	However, by regularity theory of the Stokes equations
	\[
		\|\tilde{u}_\lambda -  u_\ast\|_{L^\infty((0,T) ; W^{1,\infty}(\IR^3))} 
		\leq C \| \rho_\lambda - \rho_\ast\|_{L^\infty((0,T)\times(\IR^3))},
	\]
	where we used that by Lemma \ref{lem:APriori} we have uniform control of the support of $\rho_\lambda$.
\end{proof}

%% file: ConvergenceForArbitraryTimes.tex
\subsection{Convergence for arbitrary times}
In view of Proposition \ref{pro:strongConvergence}, it only remains to prove $T_\ast = \infty$ to finish the proof of Theorem \ref{th:main}.
The idea of the proof is the following.
Due to Lemma \ref{lem:biLipschitzV}, it is sufficient to control the quantity $M_\lambda(t)$ defined in \eqref{eq:defM} uniformly in $\lambda$.
Indeed, arguing similar as in Lemma \ref{lem:T_astPositive}, $\limsup_{\lambda \to \infty}M_\lambda(t)$ has to blow up at time $T_\ast$.
However, Proposition \ref{pro:strongConvergence} shows, that for large enough values of $\lambda$, $M_\lambda(t)$ is controlled by the corresponding quantity of the limit equation. This gives a contradiction.

\begin{proof}[Proof of Theorem \ref{th:main}]
	By Proposition \ref{pro:strongConvergence}, it suffices to prove $T_\ast = \infty$.
	By Lemma \ref{lem:T_astPositive}, we have $T_\ast >0$. Assume $T_\ast < \infty$ and let $T < T_\ast$.
	By definition of $T_\ast$ and Lemma \ref{lem:boundsU}, the assumption $\lambda \geq 4 \|\nabla u_\lambda \|_{L^\infty((0,T) \times \IR^3)}$ is
	satisfied for all $\lambda \geq \lambda_0(T)$. 
	Recall the definition of $M(T)$ from Lemma \ref{lem:biLipschitzV}, which we will now denote by $M_\lambda(T)$ to
	emphasize the dependence on $\lambda$. Moreover, we denote by $M_\ast$ the corresponding quantity for the solution of the limit problem,
	i.e.,
	\[
		M_\ast (t) := \exp \left(\int_0^t   2 \|\nabla u(s,\cdot) \|_{L^\infty(\IR^3)} \dd s \right).
	\]
	By Proposition \ref{pro:strongConvergence}, we have
	\[
		M_\lambda(T) \to M_\ast(T) \leq M_\ast(T_\ast).
	\]
	In particular, for all $\lambda \geq \lambda_0(T)$ (possibly enlarging $\lambda_0(T)$),
	\[
		M_\lambda(T) \leq 2 M_\ast(T_\ast).
	\]
	Therefore, 	Lemma \ref{lem:biLipschitzV} implies for all $\lambda \geq \lambda_0(T)$
	\[
		\sup_{s \leq t} \|\rho_\lambda(s,\cdot)\|^2_{L^\infty(\IR^3)} \leq 2 M_\ast(T_\ast).
	\]
	
	The rest of the proof is very similar to the proof of Lemma \ref{lem:T_astPositive}.
	We define 
\[
	T_\lambda := \sup \{t \geq 0 \colon \sup_{s \leq t} \|\rho_\lambda(s,\cdot)\|_{L^\infty(\IR^3)} \leq 2 C_0 (2 M_\ast(T_\ast))^3\}.  
\]
Then, $T_\lambda > T$ as $\rho_\lambda$ is continuous.
Analogously as we have shown \eqref{eq:growRho} in Lemma \ref{lem:T_astPositive}, we find that for all $t>0$ and $\lambda \geq  C \sup_{s \leq T + t} \|\rho_\lambda(s,\cdot)\|^2_{L^\infty(\IR^3)}$
\begin{equation}
	\label{eq:growRho2}
	\sup_{s \leq T + t} \|\rho_\lambda(s,\cdot)\|_{L^\infty(\IR^3)} 
	\leq C_0 (2M_\ast(T_\ast))^3 \exp \bigg(C t \sup_{s \leq T + t} \|\rho_\lambda(s,\cdot)\|^2_{L^\infty(\IR^3)}\bigg).
\end{equation}
This implies for all $\lambda \geq \max \{\lambda_0(T), C C_0^2 (M_\ast(T_\ast))^6\}$ and all $T + t < T_\lambda$
\[
	\sup_{s \leq T + t} \|\rho_\lambda(s,\cdot)\|_{L^\infty(\IR^3)} \leq C_0 (2M_\ast(T_\ast))^3 \exp( C C_0^2 (M_\ast(T_\ast))^6 t).
\]
As $\rho_\lambda$ is continuous, this yields for all $\lambda \geq \max \{ \lambda_0, C C_0^2 (M_\ast(T_\ast))^6$
\[
	T_\lambda \geq T + \frac{\log(2)}{C C_0^2 (M_\ast(T_\ast))^6}.
\]
In particular, choosing $T < T_\ast$ large enough, we deduce for all $\lambda \geq \max \{ \lambda_0, C C_0^2 (M_\ast(T_\ast))^6 \}$
\[
	T_\lambda > T_\ast,
\]
which gives a contradiction to the definition of $T_\ast$.
\end{proof}